\documentclass[10.5pt, a4paper]{amsart}
\usepackage{amssymb, amsmath, amscd, bm, amsthm, mathtools, pdfpages, setspace, subcaption, graphicx, xfrac, empheq, enumitem, array, arydshln, hyperref}

\usepackage[
backend=biber,
style=numeric-comp,
sortcites=true,
maxnames=99
]{biblatex}

\DeclareFieldFormat{pagetotal}{#1 pp}
\renewbibmacro*{note+pages}{
	\printfield{note}
	\setunit{\bibpagespunct}
	\printfield{pages}
	\setunit{\addcomma\space}
	\printfield{pagetotal}
	\clearfield{pagetotal}
}
\renewbibmacro*{chapter+pages}{
	\printfield{chapter}
	\setunit{\bibpagespunct}
	\printfield{pages}
	\setunit{\addcomma\space}
	\printfield{pagetotal}
	\clearfield{pagetotal}
}
\def\@Rref#1{\hbox{\rm \ref{#1}}}
\def\Rref#1{\@Rref{#1}}
\theoremstyle{plain}
\newtheorem{theorem}{Theorem}[section]
\newtheorem{proposition}[theorem]{Proposition}

\newtheorem{assumption}[theorem]{Assumption}
\newtheorem{lemma}[theorem]{Lemma}

\theoremstyle{definition}
\newtheorem{definition}{Definition}[section]

\newtheorem{remark}[definition]{Remark}

\newcommand{\Ls}{\mathrm{L}^{2}}
\newcommand{\Ht}{\mathcal{H}_{\tau}}
\newcommand{\Hs}{\mathrm{H}^{L}_{0}}
\newcommand{\Hso}{\mathrm{H}^{L}}
\newcommand{\Cinf}{\mathrm{C}^{\infty}}
\newcommand{\Ps}{\mathrm{P}_{L-1}}
\newcommand{\Bs}{\mathrm{B}^{L}_N}
\newcommand{\eB}{\mathrel{\overset{\mathcal{B}}{=}}}
\newcommand{\geB}{\mathrel{\overset{\mathcal{B}}{\geq}}}
\newcommand{\gB}{\mathrel{\overset{\mathcal{B}}{>}}}
\newcommand{\ggB}{\mathrel{\overset{\mathcal{B}}{\gg}}}
\newcommand{\Rsym}{\mathbb{R}^{q\times q}_{\mathrm{sym}}[\zeta,\eta]}
\DeclareMathOperator{\Ran}{Ran}
\DeclareMathOperator{\Ker}{Ker}
\DeclareMathOperator{\mathspan}{span}

\begin{document}

	\title[Data-Driven Stabilization]{Data-Driven Stabilization 
		of Continuous-Time Systems \\
		with Noisy Input-Output Data}
	
	\thispagestyle{plain}
	
	\author{Masashi Wakaiki}
	\address{Graduate School of System Informatics, Kobe University, Nada, Kobe, Hyogo 657-8501, Japan}
	\email{wakaiki@ruby.kobe-u.ac.jp}
	\thanks{This work was supported in part by 
		JSPS KAKENHI Grant Number 24K06866.}
	
	\begin{abstract}
		We study
		data-driven stabilization of continuous-time 
		systems in autoregressive form when only noisy
		input-output data
		are available.
		First, we provide an operator-based characterization
		of the set of systems consistent with the data. Next,
		combining this characterization with
		behavioral theory, we establish a necessary 
		and sufficient condition for the noisy data
		to be informative for quadratic stabilization.
		This condition is formulated in terms of linear matrix inequalities, whose solutions yield a
		stabilizing controller.
		Finally, we characterize 
		data informativity for system identification in the noise-free setting.
	\end{abstract}
	
	
	\maketitle
	
	\section{Introduction}
	\subsubsection*{Motivation and literature review}
	Data-driven control aims to design controllers for
	unknown dynamical systems from input-output data.
	Whereas model-based methods
	first construct a mathematical model and then 
	design a controller, data-driven methods operate directly on data and eliminate the intermediate identification step.
	A fundamental question in this paradigm is whether 
	the available data contain enough information to
	design controllers that achieve
	control objectives.
	To address this question, the concept of data informativity
	was introduced in \cite{Waarde2020}.
	Data informativity provides a rigorous theoretical framework to determine
	whether all systems consistent
	with the given data have the desired properties, and
	whether controllers can be designed to achieve the 
	control 
	objectives for all such data-consistent systems.
	
	Data informativity was investigated for various system
	properties and control objectives. Here
	we provide a brief review of the literature concerning stabilization; see  \cite{Waarde2023} for 
	a comprehensive introduction to the informativity approach.
	Informativity for state-feedback stabilization of discrete-time 
	systems was first characterized in \cite{Waarde2020},
	and was subsequently  extended to
	deal with process noise \cite{Waarde2022TAC,Steentjes2022,Waarde2023SIAM},
	measurement errors \cite{Bisoffi2024}, and 
	their combination \cite{Kaminaga2025arXiv}. 
	A generalization to infinite-dimensional systems
	was also proposed
	in \cite{Wakaiki2025Informativity}. 
	Furthermore,
	the case of noisy 
	input-output data
	was investigated by the behavioral approach
	in \cite{Waarde2023TAC} and 
	by the state-space approach in \cite{Steentjes2022, Li2026Automatica, Kaminaga2025arXiv}.
	Although sampled data are discrete by nature,
	many physical systems are governed by 
	continuous-time dynamics.
	Nevertheless,
	data-driven stabilization of continuous-time systems
	remains relatively unexplored compared to
	its discrete-time counterpart.
	This gap motivates us to study
	data informativity in 
	the continuous-time setting.
	
	A direct application of discrete-time techniques
	to continuous-time systems
	typically requires access to state or output derivatives
	(see, e.g., \cite{DePersis2020,Bisoffi2022,Eising2025}).
	It is worth mentioning that
	input and output derivatives play a key role
	in the continuous-time versions 
	\cite{Lopez2022, 
		Rapisarda2023SCL,
		Lopez2024, Schmitz2024}
	of the fundamental lemma developed by 
	Willems~{\em et~al.}~\cite{Willems2005}.
	However, derivative measurements
	are rarely available in practical applications.
	To circumvent this issue, several 
	derivative-free methods were proposed 
	for input-state data:
	discrete-sequence generation 
	via the integral form of the state equation \cite{DePersis2024,Song2025}
	and polynomial orthogonal bases \cite{Rapisarda2024};
	the sampling linear functional approach \cite{Ohta2024MTNS, Ohta2024,Wang2025};
	low-pass filtering \cite{Possieri2025,Bosso2025}; and 
	operator-based data embedding \cite{Wakaiki2025Cont}.
	The filtering approach was further extended
	to input-output data in \cite{Bosso2025,Bosso2025arXiv,Gao2025,Bosso2025Noisy,Possieri2026,Li2026}.
	Among these studies,
	noisy data 
	were treated for state-feedback stabilization in
	\cite{Rapisarda2024,
		Possieri2025,Wakaiki2025Cont}
	and for output-feedback stabilization in \cite{Bosso2025Noisy}.
	
\subsubsection*{Contributions}
	In this paper, we consider 
	continuous-time
	autoregressive (AR) systems described by higher-order
	differential equations.
	We assume that the system is affected by noise and that continuous-time input-output trajectories are available.
	Our objective is to design 
	a stabilizing dynamic output-feedback controller for an
	unknown system directly from the noisy data. 
	The main contributions 
	of this work are threefold:
	\begin{enumerate}\vspace{4pt}
		\renewcommand{\labelenumi}{\textup{\alph{enumi})}}
		\item We propose an operator-based data-embedding 
		technique to characterize the set of all systems
		consistent with the noisy data.
		\item 
		We derive
		a necessary and sufficient condition
		for the noisy data to be informative for
		quadratic stabilization.
		This condition is formulated in terms of linear matrix 
		inequalities (LMIs), and 
		a stabilizing controller can be constructed 
		from a feasible solution of the LMIs.
		\item 
		As a byproduct of our stabilization analysis,
		we provide a matrix-based 
		characterization of data informativity 
		for system identification in the noise-free setting.
	\end{enumerate}\vspace{4pt}
	The core technical novelty lies in establishing a characterization of data informativity that, although rooted in the operator-theoretic framework, can  be verified through matrix computations.
	This is achieved by exploiting the finite-rank structure of the operators used for the data embedding.
	
	\subsubsection*{Comparison with related work}
	The main advantage of our approach 
	over the existing methods for
	derivative-free output-feedback control 
	\cite{Bosso2025,Bosso2025arXiv,Gao2025,Bosso2025Noisy,Possieri2026,Li2026} is that 
	it directly characterizes the
	informativity of the original continuous-time data.
	Existing methods employ sampled or filtered data representations. The resulting conditions are typically 
	sufficient for controller design
	and depend on
	tuning parameters introduced by
	sampling or filtering.
	In the proposed method, we embed continuous-time data
	into synthesis operators by extending  the approach
	proposed for state-feedback stabilization in \cite{Wakaiki2025Cont}.
	This operator-theoretic approach 
	allows a seamless extension
	of data-matrix techniques developed for discrete-time systems
	to the continuous-time setting.
	In our framework, both necessity and sufficiency 
	for informativity can be addressed in terms of 
	the original continuous-time data
	without introducing sampling or filtering parameters.
	The synthesis operators used in this study 
	are variants of those in frame theory \cite{Christensen2016} and 
	are closely related to the modulating 
	function method \cite{Shinbrot1954} for system identification; see also 
	the survey article \cite{Preisig1993} for 
	this identification method.
	
	Another important distinction lies in the data setting.
	The existing studies 
	assume that the available data are given as a single
	input-output trajectory.
	In contrast, our framework explicitly accommodates 
	multiple
	input-output trajectories.
	This feature is particularly relevant for data-driven 
	stabilization
	of unstable systems, 
	where obtaining a single long trajectory 
	may be impractical because 
	the response may grow rapidly.
	Even in such scenarios, 
	the proposed method can be used to design a stabilizing
	controller from a collection of short input-output trajectories.
	
	Our approach also differs from the
	those in the existing literature in terms of model representation. While
	existing methods are based on state-space
	realizations,
	we employ a continuous-time AR 
	representation. This choice allows us to
	extend the synthesis-operator
	approach in \cite{Wakaiki2025Cont} to output-feedback stabilization,
	as in the discrete-time setting~\cite{Waarde2023TAC}.
	This treatment is enabled by
	the behavioral framework~\cite{Willems1986a, Willems1986b,
		Willems1987} and the associated concept
	of quadratic differential forms~\cite{Willems1998}.
	
	\subsubsection*{Organization}
	This paper is organized as follows.
	In Section~\ref{sec:adjoints}, 
	we introduce higher-order synthesis operators and
	provide an integral representation for
	products of synthesis operators and their
	adjoints.
	In Section~\ref{sec:operator_based_embedding},
	we characterize the set of all systems consistent
	with the given noisy data in terms of synthesis operators.
	In Section~\ref{sec:stabilization},
	the main result on
	data-driven stabilization is presented.
	In Section~\ref{sec:surjectivity},
	we briefly investigate the informativity of noise-free data
	for system identification.
	A simulation example illustrating the proposed method is given in
	Section~\ref{sec:example}. 
	Finally, Section~\ref{sec:conclusion} offers concluding remarks.
	
	\subsubsection*{Notation.}
	We denote by $\mathbb{N}$ the set of 
	positive integers and define $\mathbb{N}_0 \coloneqq 
	\mathbb{N} \cup \{ 0\}$.
	We denote by $\mathbb{S}^n$
	the set of $n \times n$ real symmetric matrices.
	For a matrix $A$,
	its transpose is denoted by
	$A^{\top}$.
	For $A \in \mathbb{S}^n$, 
	we write $A > 0$ (resp. $A \geq 0$)
	if $A$ is positive definite (resp. nonnegative definite).
	Analogous notation is used for 
	negative and nonpositive definite matrices.
	The $n\times n$ identity 
	matrix and 
	the $n\times m$ 
	zero matrix are denoted by $I_n$ and $0_{n\times m}$,
	respectively.
	We denote by
	$\langle \cdot,\cdot \rangle_{\mathbb{R}^n}$
	the standard inner product on the Euclidean
	space $\mathbb{R}^n$.
	
	Let $\tau >0$ and $L \in \mathbb{N}$.
	We denote by $\Ls([0,\tau]; \mathbb{R}^n)$
	the space of measurable functions $f\colon [0,\tau]\to \mathbb{R}^n$ satisfying $\int_0^\tau \|f(t)\|_{\mathbb{R}^n}^2 dt < \infty$.
	The Sobolev space $\Hso([0,\tau]; \mathbb{R}^n)$
	consists of all functions
	$\phi \in \Ls([0,\tau]; \mathbb{R}^n)$ such that $\phi,\phi',\dots, \phi^{(L-1)}$
	are absolutely continuous  on $[0,\tau]$ and 
	$\phi^{(L)} \in \Ls([0,\tau]; \mathbb{R}^n)$.
	We denote by $\Cinf(\mathbb{R};\mathbb{R}^n)$ 
	the space of infinitely differentiable functions
	from $\mathbb{R}$ to $\mathbb{R}^n$.
	For brevity,
	we write 
	$\Ls[0,\tau] \coloneqq 
	\Ls([0,\tau]; \mathbb{R})$, with $\Hso[0,\tau]$
	and $ \Cinf(\mathbb{R})$ defined similarly.
	We define the subspace
	$\Hs[0,\tau]$ of $\Hso[0,\tau]$ by
	\[
	\Hs[0,\tau] \coloneqq 
	\{
	\phi \in \Hso[0,\tau]:
	\phi^{(\ell)}(0)=0 \text{~and~}
	\phi^{(\ell)}(\tau) = 0 \text{~for all~} \ell=0,\dots,L-1
	\}
	\]
	with inner product
	\[
	\langle 
	\phi, \psi 
	\rangle_{\Hs} \coloneqq
	\int_0^{\tau} \phi^{(L)}(t) \psi^{(L)}(t)dt,\quad 
	\phi,\psi \in \Hs[0,\tau].
	\]

	Let $X$ and $Y$ be Hilbert spaces. We denote by
	$\mathcal{L}(X,Y)$ the space of bounded linear 
	operators from $X$ to $Y$.
	We write $\mathcal{L}(X) \coloneqq \mathcal{L}(X,X)$.
	For $T \in \mathcal{L}(X,Y)$, 
	its range and 
	kernel are denoted by
	$\Ran T$ and $\Ker T$, respectively.
	The Hilbert-space adjoint of $T$ is denoted by $T^*$.
	For a subset $M$ of $X$,
	its orthogonal complement
	is denoted by $M^{\perp}$.

	\section{Higher-order synthesis operators}
	\label{sec:adjoints}
	As in the state-feedback case~\cite{Wakaiki2025Cont},
	we embed continuous-time data into synthesis operators, which 
	are continuous-time analogues of the data matrices used in the discrete-time setting. 
	For continuous-time AR systems, we require
	higher-order synthesis operators involving higher-order derivatives.
	\begin{definition}
		\label{def:synthesis}
		Let $\tau>0$ and $L \in \mathbb{N}$.
		For $f \in \Ls([0,\tau];\mathbb{R}^n)$ 
		and $\ell =0,\dots,L$,
		we define 
		$F_\ell \in \mathcal{L}(\Hs[0,\tau]; \mathbb{R}^n)$
		by
		\[
		F_{\ell} \phi \coloneqq 
		(-1)^{\ell} \int_0^{\tau} \phi^{(\ell)}(t) 
		f(t) dt,\quad \phi \in \Hs[0,\tau].
		\]
		We call $F_{\ell}$ the {\em $\ell$-th
			synthesis operator associated with $f$}.
		We simply call $F_0$ the {\em synthesis operator
			associated with $f$}.
	\end{definition}
	
	Since $F_j$ and $F_{\ell}$ are finite-rank operators (i.e.,
	their ranges have finite dimension),
	we can identify the product $F_jF_{\ell}^*$
	with an $n \times n$ matrix.
	Note also that $F_jF_{\ell}^*$
	is determined only by $f$ and 
	is independent of any particular choice of 
	$\phi \in \Hs[0,\tau]$
	used in Definition~\ref{def:synthesis}.
	The synthesis-operator approach 
	is based on such products of synthesis operators and
	their adjoints. 
	In the remainder of this section, we derive an explicit formula
	for these products. 
	To this end, 
	we first present two technical lemmas in Section~\ref{sec:preliminary_sync} and then
	obtain an integral representation of 
	the products in Section~\ref{sec:integra_rep}.
	
	\subsection{Preliminary lemmas}
	\label{sec:preliminary_sync}
	Let $\tau >0$ and 
	define 
	the integral operator $\mathcal{J}\in \mathcal{L}(\Ls[0,\tau])$ 
	by
	\begin{equation}
		\label{eq:int_op}
		(\mathcal{J}f)(t) \coloneqq \int_0^t f(s)ds,\quad 
		t \in [0,\tau].
	\end{equation}
	It follows by induction and Fubini's theorem that
	\begin{equation}
		\label{eq:l_times_integral}
		(\mathcal{J}^\ell f)(t) = \frac{1}{(\ell-1)!}
		\int_0^t (t-s)^{\ell-1} f(s)ds
	\end{equation}
	for all $\ell \in \mathbb{N}$, $f \in \Ls[0,\tau]$, and $t \in [0,\tau]$.
	For $L \in \mathbb{N}_0$,
	we denote by
	$\mathrm{P}_L[0,\tau]$
	the space of  polynomials 
	of degree at most $L$ on the interval $[0,\tau]$.
	The first lemma shows that the orthogonal complement in $\Ls[0,\tau]$ of the subspace 
	$\{ 
	\phi^{(L)} : \phi \in \Hs[0,\tau]
	\}$ 
	coincides with $\Ps[0,\tau]$.
	\begin{lemma}
		\label{lem:orthogonal_comp}
		For all $L \in \mathbb{N}$, one has
		$\Ps[0,\tau]  = 
		\{ 
		\phi^{(L)} : \phi \in \Hs[0,\tau]
		\}^{\perp}$, where the orthogonal complement
		is taken with respect to the standard 
		inner product on
		$\Ls[0,\tau]$.
	\end{lemma}
	\begin{proof}
		Let $f \in\Ps[0,\tau]$. Then
		$f^{(L)} = 0$. Hence, for all $\phi \in \Hs[0,\tau]$,
		integrating by parts $L$ times yields
		\[
		\int_0^{\tau} \phi^{(L)}(t) f(t)dt = (-1)^L
		\int_0^{\tau} \phi(t) f^{(L)}(t)dt = 0.
		\]
		This implies that $f \in \{ 
		\phi^{(L)} : \phi \in \Hs[0,\tau]
		\}^{\perp}$.
		
		To show the converse inclusion,
		let $g\in
		(\Ps[0,\tau])^{\perp}$, and define
		$\psi  \coloneqq \mathcal{J}^Lg$, where
		the operator $\mathcal{J}$ is as in \eqref{eq:int_op}.
		Then $\psi \in \Hso[0,\tau]$ and 
		$\psi^{(\ell)} = \mathcal{J}^{L-\ell}g$ for all $\ell=0,\dots,L$. 
		By the definition of $\mathcal{J}$, 
		\[
		\psi^{(\ell)}(0) = (\mathcal{J}^{L-\ell}g)(0)= 0
		\] 
		for all $\ell =0,\dots,L-1$.
		From $g\in
		(\Ps[0,\tau])^{\perp}$ and \eqref{eq:l_times_integral},
		it also follows that
		for all $\ell =0,\dots,L-1$, 
		\[
		\psi^{(\ell)}(\tau) = (\mathcal{J}^{L-\ell}g)(\tau) = 
		\frac{1}{(L-\ell-1)!} 
		\int_0^{\tau} (\tau-s)^{L-\ell-1} g(s) ds = 0.
		\]
		Therefore, $\psi \in \Hs[0,\tau]$.
		Noting that $g= \psi^{(L)}$, we have
		\begin{equation}
			\label{eq:P_perp_inclusion}
			(\Ps[0,\tau])^{\perp}  \subseteq 
			\{ 
			\phi^{(L)} : \phi \in \Hs[0,\tau]
			\}.
		\end{equation}
		Since $\Ps[0,\tau]$ is a finite-dimensional space, it is closed.
		Hence,
		$((\Ps[0,\tau])^{\perp})^{\perp} = \Ps[0,\tau]$.
		Combining this with \eqref{eq:P_perp_inclusion},
		we obtain the desired inclusion.
	\end{proof}
	Define the
	polynomial vector $\lambda_j$ by
	\begin{equation}
		\label{eq:Gamk_def}
		\lambda_j(t) \coloneqq 
		\renewcommand{\arraystretch}{2.3}
		\begin{bmatrix}
			\dfrac{0!}{j!} t^{j} \\
			\dfrac{1!}{(j+1)!} t^{j+1} \\
			\vdots \\
			\dfrac{(L-1)!}{(j+L-1)!} t^{j+L-1}
		\end{bmatrix},\quad j=0,\dots,L.
	\end{equation}
	Define the $L \times L$ polynomial matrix $\Lambda$ by
	\begin{align}
		\label{eq:Gam_def}
		\Lambda(t) &\coloneqq 
		\begin{bmatrix}
			\lambda_1(t) & \cdots & 
			\lambda_L(t)
		\end{bmatrix}^{\top}.
	\end{align}
	Since 
	\[
	\int_0^1 (1-s)^{j-1} s^{\ell-1}  ds
	=  \frac{(j-1)!(\ell-1)!}{(j+\ell-1)!}
	\]
	for all $j,\ell \in \mathbb{N}$ (see, e.g., \cite[Fact~14.2.1]{Bernstein2018}), 
	we obtain
	\begin{equation}
		\label{eq:beta_func}
		\frac{1}{(j-1)!}
		\int_0^{t} (t-s)^{j-1} \lambda_0(s) ds = 
		\lambda_j(t)
	\end{equation}
	for all $j=1,\dots,L$.

	The next lemma will be used to derive
	a formula for the adjoints of synthesis operators.
	\begin{lemma}
		\label{lem:Gam_inv}
		Define $\Lambda(t) \in \mathbb{R}^{L\times L}$
		by \eqref{eq:Gam_def}. Then
		$\Lambda(t)$ is invertible for all $t>0$.	
	\end{lemma}
	\begin{proof}
		Fix $t>0$, and 
		suppose that $\Lambda(t)b = 0$ for some
		$b \in \mathbb{R}^{L}$.
		Then
		\begin{equation}
			\label{eq:gamma_ib0}
			\lambda_j(t)^{\top} b = 0
		\end{equation}
		for all $j=1,\dots,L$.
		By \eqref{eq:beta_func} and \eqref{eq:gamma_ib0},
		\[
		\int_0^t (t-s)^{j-1} \lambda_0(s)^{\top} bds = (j-1)! \lambda_{j}(t)^{\top}b = 0
		\]
		for all $j=1,\dots,L$.
		Since $\Ps[0,t]$ is spanned by
		$\{(t-s)^{j-1}:j=1,\dots,L \}$,
		we have 
		\begin{equation}
			\label{eq:fgb_int}
			\int_0^t f(s) \lambda_0(s)^{\top}b  ds = 0
		\end{equation}
		for all $f \in \Ps[0,t]$.
		Substituting 
		$f(s) = \lambda_0(s)^{\top}b$
		into \eqref{eq:fgb_int},
		we derive 
		\[
		\lambda_0(s)^{\top}b= 0
		\]
		for all $s \in [0,t]$,
		which implies that $b = 0$. Thus, $\Lambda(t)$ is invertible.
	\end{proof}
	
	\subsection{Products
		of synthesis operators and their adjoints}
	\label{sec:integra_rep}
	The following proposition provides an 
	explicit representation
	of the adjoints of synthesis operators
	and a formula for the corresponding products.
	\begin{proposition}
		\label{prop:adjoint}
		Let $\tau >0$ and $L \in \mathbb{N}$.
		For $\ell =0,\dots,L$,
		let  $F_{\ell}\in 
		\mathcal{L}(\Hs[0,\tau],\mathbb{R}^n)$
		and $G_{\ell}\in 
		\mathcal{L}(\Hs[0,\tau],\mathbb{R}^m)$ be the $\ell$-th
		synthesis operators associated with $f\in \Ls([0,\tau];\mathbb{R}^n)$ and $g\in \Ls([0,\tau];\mathbb{R}^m)$, respectively.
		Define $\lambda_j(t)\in \mathbb{R}^{L}$ and $\Lambda(t) \in \mathbb{R}^{L\times L}$ by
		\eqref{eq:Gamk_def} and \eqref{eq:Gam_def}, respectively.
		Define
		$\widetilde{f}_\ell \in \Ls([0,\tau];\mathbb{R}^n)$
		and $\Theta_{\ell} \in \mathbb{R}^{L \times n}$ by
		\begin{align*}
			\widetilde{f}_\ell(t) &\coloneqq 
			\begin{cases}
				(-1)^L f(t),& \ell=L, \\[4pt]
				\displaystyle
				\frac{(-1)^{L}}{(L-\ell-1)!} \int_0^t (t-s)^{L-\ell-1} f(s)ds,
				&\ell=-L,-L+1,\dots,L-1,
			\end{cases} \\
			\Theta_{\ell} &\coloneqq \Lambda(\tau)^{-1}
			\begin{bmatrix}
				\widetilde{f}_{\ell-1}(\tau) & 
				\widetilde{f}_{\ell-2}(\tau) & \cdots &
				\widetilde{f}_{\ell-L}(\tau)
			\end{bmatrix}^{\top},\quad \ell=0,\dots,L.
		\end{align*}
		Then the following statements hold:
		\begin{enumerate}
			\renewcommand{\labelenumi}{\textup{\alph{enumi})}}
			\item 
			For all $\ell=0,\dots,L$, $b \in \mathbb{R}^n$,
			and $t \in [0,\tau]$,
			the adjoint $F_{\ell}^* \in 
			\mathcal{L}(\mathbb{R}^n, \Hs[0,\tau])$
			satisfies
			\begin{equation}
				\label{eq:Tl_adjoint}
				(F_{\ell}^*b)(t) = 
				(\widetilde{f}_{\ell-L}(t)^{\top} - \lambda_L(t)^{\top}\Theta_{\ell} )b.
			\end{equation}
			\item
			For all $j,\ell=0,\dots,L$, 
			the matrix $G_jF_\ell^* \in \mathbb{R}^{m\times n}$
			is given by
			\begin{equation}
				\label{eq:SjTl}
				G_jF_\ell^* =
				(-1)^j
				\int_0^{\tau} g(t) 
				( \widetilde{f}_{j+\ell -L} (t)^{\top} - 
				\lambda_{L-j}(t)^{\top}\Theta_{\ell}
				)dt.
			\end{equation}
		\end{enumerate}
	\end{proposition}
	\begin{proof}	
		a)
		Let $\ell =0,\dots,L$, $b \in \mathbb{R}^n$, 
		and $\phi \in \Hs[0,\tau]$.
		Define
		$f_b(t) \coloneqq (-1)^{L}f(t)^{\top}b$  for
		$t \in [0,\tau]$.
		Integrating by parts $L-\ell$ times, we can write
		the inner product of $F_{\ell} \phi$ and $b$,
		defined by
		$
		\langle 
		F_{\ell} \phi, b
		\rangle_{\mathbb{R}^n} \coloneqq 
		(F_{\ell} \phi)^{\top} b$,
		as
		\[
		\langle 
		F_{\ell} \phi, b
		\rangle_{\mathbb{R}^n}=
		\int_0^{\tau} \phi^{(L)}(t) (\mathcal{J}^{L-\ell}f_b)(t)  dt,
		\]
		where the operator $\mathcal{J}$
		is as in \eqref{eq:int_op}.
		By the definition of adjoints, we also have
		\[
		\langle 
		F_{\ell} \phi, b
		\rangle_{\mathbb{R}^n} =
		\langle 
		\phi, F_{\ell}^{*}b
		\rangle_{\Hs} = \int_0^{\tau} \phi^{(L)}(t) 
		(F_{\ell}^{*}b)^{(L)}(t)dt.
		\]
		Therefore, we obtain
		\[
		\int_0^{\tau} \phi^{(L)}(t) \left(
		(F_{\ell}^{*}b)^{(L)}(t) -  (\mathcal{J}^{L-\ell}f_b)(t)
		\right)
		dt = 0.
		\]
		
		By Lemma~\ref{lem:orthogonal_comp},
		there exists $\gamma \in \Ps[0,\tau]$ such that
		\begin{equation}
			\label{eq:L_derivative}
			(F_{\ell}^{*}b)^{(L)}(t) =  (\mathcal{J}^{L-\ell}f_b)(t)
			+\gamma(t)
		\end{equation}
		for a.e.~$t \in [0,\tau]$.
		Since $(F_{\ell}^{*}b)^{(j)}(0) = 0$ for all $j=0,\dots,L-1$, 
		the fundamental theorem of calculus shows that 
		\begin{equation}
			\label{eq:n_derivative}
			(F_{\ell}^{*}b)^{(j)}(t) = 
			(\mathcal{J}^{2L-j-\ell} f_b)(t) + (\mathcal{J}^{L-j}\gamma)(t)
		\end{equation}
		for all $t \in [0,\tau]$ and $j =0,\dots,L-1$.
		Choose $b_0 \in \mathbb{R}^L$ such that 
		its $j$-th element is 
		the coefficient of $t^{j-1}$ in the polynomial $\gamma(t)$.
		Then 
		\[
		\gamma(t) = \lambda_0(t)^{\top} b_0
		\]
		for all $t \in [0,\tau]$, where 
		$\lambda_0(t)$ is as in \eqref{eq:Gamk_def}.
		By \eqref{eq:l_times_integral} and  
		\eqref{eq:beta_func},
		\begin{equation}
			\label{eq:EfEt_expression}
			(\mathcal{J}^{2L-j-\ell} f_b)(t) + (\mathcal{J}^{L-j}\gamma)(t) 
			=
			\widetilde{f}_{j+\ell-L}(t)^{\top}b + 
			\lambda_{L-j}(t)^{\top}b_0
		\end{equation}
		for all $t \in [0,\tau]$ and 
		$j=0,\dots,L-1$. 
		Combining \eqref{eq:n_derivative} with
		\eqref{eq:EfEt_expression}, we obtain
		\begin{equation}
			\label{eq:n_derivative2}
			(F_{\ell}^{*}b)^{(j)}(t) = 
			\widetilde{f}_{j+\ell-L}(t)^{\top}b + 
			\lambda_{L-j}(t)^{\top}b_0
		\end{equation}
		for all $t \in [0,\tau]$ and 
		$j=0,\dots,L-1$.
		
		We now show that 
		\begin{equation}
			\label{eq:c_expression}
			b_0 = -\Theta_{\ell}b.
		\end{equation}
		Since
		$(F_{\ell}^{*}b)^{(j)}(\tau) = 0$ for all $j=0,\dots,L-1$,
		we deduce from \eqref{eq:n_derivative} that
		\begin{equation}
			\label{eq:tau_boundary}
			(\mathcal{J}^{2L-j-\ell} f_b)(\tau) + (\mathcal{J}^{L-j}\gamma)(\tau)= 0
		\end{equation}
		for all $j = 0,\dots,L-1$.
		From \eqref{eq:EfEt_expression} and \eqref{eq:tau_boundary},
		it follows that
		\[
		\lambda_{L-j}(\tau)^{\top}b_0= - \widetilde{f}_{j+\ell-L}(\tau)^{\top}b 
		\]
		for all $j = 0,\dots,L-1$.
		Since $\Lambda(\tau)$ is invertible by Lemma~\ref{lem:Gam_inv},
		we conclude that \eqref{eq:c_expression} holds.

		By \eqref{eq:n_derivative2} with 
		$j=0$,
		we obtain
		\begin{equation}
			\label{eq:Tell_b}
			(F_{\ell}^*b)(t)
			= 
			\widetilde{f}_{\ell-L}(t)^{\top}b + 
			\lambda_L(t)^{\top}b_0
		\end{equation}
		for all $t \in [0,\tau]$.
		The assertion \eqref{eq:Tl_adjoint} is an immediate consequence of \eqref{eq:c_expression} and \eqref{eq:Tell_b}.
		
		b)
		Let $j,\ell \in \{0,\dots,L\}$ and $b \in \mathbb{R}^n$.
		By definition,
		\begin{equation}
			\label{eq:GjFl*}
			G_j F_{\ell}^* b = 
			(-1)^j
			\int_0^{\tau} (F_{\ell}^*b)^{(j)}(t) 
			g(t) dt.
		\end{equation}
		If $0 \leq j \leq L-1$, then
		\eqref{eq:n_derivative2}
		and \eqref{eq:c_expression} yield
		\begin{align}
			\label{eq:Tlb}
			(F_{\ell}^*b)^{(j)}(t) 
			&=
			(\widetilde{f}_{j+\ell-L}(t)^{\top} - \lambda_{L-j}(t)^{\top} \Theta_{\ell})b
		\end{align}
		for all $t \in [0,\tau]$.
		It follows from \eqref{eq:L_derivative}
		that 
		\eqref{eq:Tlb} is also valid for a.e.~$t \in [0,\tau]$
		in the case $j=L$.
		Substituting \eqref{eq:Tlb} into \eqref{eq:GjFl*},
		we obtain
		\begin{align*}
			G_jF_{\ell}^{*}b &=
			(-1)^j\int_0^{\tau} 
			\big( (\widetilde{f}_{j+\ell-L}(t)^{\top} - \lambda_{L-j}(t)^{\top}\Theta_{\ell} )b \big)g(t)dt  \\
			&=
			\left( (-1)^j 
			\int_0^{\tau}  g(t)
			( 
			\widetilde{f}_{j+\ell-L}(t)^{\top} - \lambda_{L-j}(t)^{\top}\Theta_{\ell}
			) dt 
			\right)b. 
		\end{align*}
		Thus, 
		the integral representation \eqref{eq:SjTl}
		of $G_jF_\ell^*$ is obtained.
	\end{proof}

	\section{Operator-based data embedding}
	\label{sec:operator_based_embedding}
	In this section, we develop
	an operator-based method for data embedding.
	We first describe the input-output data considered in this paper and introduce the associated synthesis operators in Section~\ref{sec:data_synthesis_operators}.
	We then characterize the 
	set of systems consistent with
	the given data in Section~\ref{sec:characterization_data_consistency}.
	\subsection{System and data}
	\label{sec:data_synthesis_operators}
	Consider the following continuous-time AR system of order $L \in \mathbb{N}$:
	\begin{equation}
		\label{eq:system_original}
		y^{(L)}(t) + \sum_{\ell=0}^{L-1}
		\left(
		R_{u,\ell} u^{(\ell)}(t)  + R_{y,\ell} y^{(\ell)}(t) 
		\right)
		= v(t),\quad t\geq 0,
	\end{equation}
	where $R_{u,\ell} \in \mathbb{R}^{p \times m}$
	and $R_{y,\ell} \in \mathbb{R}^{p\times p}$
	for $\ell=0,\dots,L-1$. In the system
	\eqref{eq:system_original},
	$u(t) \in \mathbb{R}^m$,
	$y(t) \in \mathbb{R}^p$, and 
	$v(t) \in \mathbb{R}^p$ are 
	the input, the output, and the noise at time $t$,
	respectively.
	Let $q \coloneqq m+p$, and
	define $R \in \mathbb{R}^{p\times qL}$ by
	\begin{equation}
		\label{eq:R_def}
		R \coloneqq 
		\begin{bmatrix}
			R_{u,0} & R_{y,0} & R_{u,1} & R_{y,1} & \cdots & R_{u,L-1} & R_{y,L-1}
		\end{bmatrix}.
	\end{equation}
	Then the system \eqref{eq:system_original} can be written as
	\begin{equation}
		\label{eq:system}
		y^{(L)}(t) + R
		\begin{bmatrix}
			u(t) \\ y(t)  \\ \vdots \\ u^{(L-1)}(t) \\ y^{(L-1)}(t)
		\end{bmatrix} = v(t),\quad t\geq 0.
	\end{equation}
	Since the matrix $R$ uniquely determines the system \eqref{eq:system},
	we identify the matrix $R$ with the system \eqref{eq:system}
	throughout this paper.
	We consider a scenario where the system order $L$ 
	is known, whereas 
	the matrix $R$ and the noise $v$ are unknown.
	Input-output data generated by the system \eqref{eq:system} 
	are available, as assumed in the 
	existing studies \cite{Bosso2025,Bosso2025arXiv,Gao2025,Bosso2025Noisy,Possieri2026,Li2026}.
	
	Let $K \in \mathbb{N}$ and 
	$\tau_1,\dots,\tau_K >0$. 
	For $k=1,\dots,K$, 
	we denote 
	the $k$-th input trajectory by $u_k$ and 
	the $k$-th output trajectory by $y_k$. 
	Since \eqref{eq:system}
	involves the $(L-1)$-th derivative of the input and
	the $L$-th  derivative of the output,
	we assume that 
	$u_k \in \mathrm{H}^{L-1}([0,\tau_k];\mathbb{R}^m)$ 
	and $y_k \in \Hso([0,\tau_k];\mathbb{R}^p)$ for all $k=1,\dots,K$.
	Let
	$v_k \in \Ls([0,\tau_k];\mathbb{R}^p)$ be the noise corrupting
	the $k$-th trajectory.
	We expect that an
	unknown true system $R \in \mathbb{R}^{p\times qL}$ satisfies 
	\begin{equation}
		\label{eq:system_k}
		y_k^{(L)}(t) + R
		\begin{bmatrix}
			u_k(t) \\ y_k(t)  \\ \vdots \\ u_k^{(L-1)}(t) \\ y_k^{(L-1)}(t)
		\end{bmatrix} = v_k(t)\quad \text{for a.e.~$t \in[0,\tau_k]$
			and all $k=1,\dots, K$}.
	\end{equation}
	
	Let $\tau \coloneqq (\tau_k)_{k=1}^K$.
	To simplify the notation, we define
	the data set $\Gamma_{\tau}$ and 
	the noise set $\Delta_{\tau}$ by
	\begin{align*}
		\Gamma_{\tau} &\coloneqq
		\{
		(u_k,y_k)_{k=1}^{K}: u_k \in \mathrm{H}^{L-1}([0,\tau_k];\mathbb{R}^m) \\
		&\hspace{90pt}\text{~and~}y_k \in \Hso([0,\tau_k];\mathbb{R}^p) \text{~for all~}k=1,\dots,K
		\},\\
		\Delta_{\tau} &\coloneqq
		\{
		(v_k)_{k=1}^K: v_k \in \Ls([0,\tau_k];\mathbb{R}^p)
		\text{~for all~}k=1,\dots,K
		\}.
	\end{align*}
	We denote by
	$\mathcal{H}_{\tau}$
	the orthogonal direct sum
	of the spaces $(\Hs[0,\tau_k])_{k=1}^K$; i.e.,
	$\Ht$ is defined by
	\begin{equation}
		\label{eq:Ht_def}
		\Ht \coloneqq 
		\{
		(\phi_k)_{k=1}^K: \phi_k \in \Hs[0,\tau_k] \text{~for all $k=1,\dots,K$}
		\}
	\end{equation}
	with the standard inner product given by the sum of the 
	componentwise $\Hs$-inner products.
	
	\subsection{Characterization of the set of data-consistent
		systems}
	\label{sec:characterization_data_consistency}
	Let $\tau = (\tau_k)_{k=1}^K$ with $\tau_1,\dots,\tau_K>0$.
	Let $(u_k,y_k)_{k=1}^K \in 
	\Gamma_{\tau}
	$ be the available input-output data, and 
	let $(v_k)_{k=1}^K \in \Delta_{\tau}$
	be the associated noise.
	Throughout this paper, 
	we use the following synthesis operators
	for $\ell=0,\dots,L-1$ and
	$k=1,\dots,K$:\vspace{4pt}
	\begin{enumerate}
		\renewcommand{\labelenumi}{\textup{\alph{enumi})}}
		\item 
		$W_{\ell,k}$ is the 
		$\ell$-th synthesis operator associated with
		the $k$-th input-output pair
		$\begin{bmatrix}
			u_k \\ y_k
		\end{bmatrix}$.
		\item $Y_{L,k}$ is the $L$-th synthesis operator
		associated with the $k$-th output $y_k$.
		\item 
		$V_{0,k}$ is
		the
		synthesis operator associated with the $k$-th noise
		$v_k$.
	\end{enumerate}\vspace{4pt}
	Using the function space $\Ht$ defined by
	\eqref{eq:Ht_def}, we define the operators $H \in \mathcal{L}(\Ht,\mathbb{R}^{qL})$ and 
	$Y_L,V_0  \in \mathcal{L}(\Ht, \mathbb{R}^p)$ by
	\begin{align}
		\label{eq:H_def}
		H \phi &\coloneqq 
		\sum_{k=1}^{K}
		\begin{bmatrix}
			W_{0,k}\phi_k \\ 
			\vdots \\
			W_{L-1,k}\phi_k
		\end{bmatrix}, \\
		Y_{L} \phi &\coloneqq 
		\sum_{k=1}^{K}Y_{L,k}\phi_k, \label{eq:YL_def} \\
		\label{eq:V0_def}
		V_0 \phi &\coloneqq 
		\sum_{k=1}^{K}V_{0,k}\phi_k
	\end{align}
	for $\phi = (\phi_k)_{k=1}^K \in \Ht$.
	The operator $H$ plays the same role as
	a data Hankel matrix
	employed in the discrete-time setting~\cite{Waarde2023TAC}.
	For $k=1,\dots,K$, we sometimes use the operator
	$H_k \in \mathcal{L}(\Hs[0,\tau_k],\mathbb{R}^{qL})$
	associated only with the $k$-th input-output pair, which
	is defined by
	\begin{equation}
		\label{eq:Hi_def}
		H_k\phi_k  \coloneqq 
		\begin{bmatrix}
			W_{0,k}\phi_k    \\ \vdots \\ W_{L-1,k}\phi_k 
		\end{bmatrix},\quad \phi_k 
		\in \Hs[0,\tau_k]
	\end{equation}
	
	We provide a data-consistency condition in terms of 
	the operators $H$, $Y_L$, and $V_0$, which extends the state-feedback result for single-trajectory data~\cite[Lemma~2.3]{Wakaiki2025Cont}.
	\begin{lemma}
		\label{lem:data_consistency_cond}
		Let $(u_k,y_k)_{k=1}^K \in 
		\Gamma_{\tau}
		$ and $(v_k)_{k=1}^K \in \Delta_{\tau}$.
		Let the operators $H$, $Y_L$ and $V_0$ be as above.
		Then
		the following statements are equivalent
		for 
		all $R \in \mathbb{R}^{p\times qL}$:
		\begin{enumerate}
			\renewcommand{\labelenumi}{\textup{(\roman{enumi})}}
			\item 
			The differential equations given in \eqref{eq:system_k} hold.
			\item 
			The operator equation $RH + Y_L = V_0$ holds.
		\end{enumerate} 
	\end{lemma}
	\begin{proof}
		Let 
		$R  \in \mathbb{R}^{p\times qL}$ be partitioned as in \eqref{eq:R_def}, and 
		let $\phi = (\phi_k)_{k=1}^K \in \Ht$ be arbitrary.
		For all $\ell=0,\dots,L-1$ and $k=1,\dots,K$,
		integrating by parts $\ell$ times yields
		\begin{equation*}
			W_{\ell,k} \phi_k = 
			\int_0^{\tau_k} \phi_k(t) 
			\begin{bmatrix}
				u_k^{(\ell)}(t) \\ y_k^{(\ell)}(t)
			\end{bmatrix}dt.
		\end{equation*}
		This implies that
		\begin{align}
			RH\phi & = 
			\sum_{\ell=0}^{L-1}
			\begin{bmatrix}
				R_{u,\ell} & R_{y,\ell }
			\end{bmatrix} 
			\left ( \sum_{k=1}^K 
			W_{\ell,k}\phi_k \right) \notag \\
			&=
			\sum_{k=1}^K 
			\int_0^{\tau_k} \phi_k(t)
			\sum_{\ell=0}^{L-1}
			\left(
			R_{u,\ell} u_k^{(\ell)}(t) + R_{y,\ell} y_k^{(\ell)}(t)
			\right)dt.
			\label{eq:RH}
		\end{align}
		It follows that
		$RH\phi +Y_{L}\phi = V_{0}\phi $ if and only if
		\begin{equation*}
			\sum_{k=1}^K 
			\int_0^{\tau_k} \phi_k (t)
			\left(y_k^{(L)}(t) + \sum_{\ell=0}^{L-1}
			\left(
			R_{u,\ell} u_k^{(\ell)}(t) + R_{y,\ell} y_k^{(\ell)}(t)
			\right) - v_k(t)\right) dt = 0.
		\end{equation*}
		From this, the implication (i) $\Rightarrow$ (ii)
		follows immediately.
		
		To prove the implication (ii) $\Rightarrow$ (i),
		let $k \in \{1,\dots,K \}$ be arbitrary, and take 
		$\phi_j = 0$ for $j\neq k$. 
		Then statement (ii) implies that
		\[
		\int_0^{\tau_k} \phi_k (t)
		\left(y_k^{(L)}(t) + \sum_{\ell=0}^{L-1}
		\left(
		R_{u,\ell} u_k^{(\ell)}(t) + R_{y,\ell} y_k^{(\ell)}(t)
		\right) - v_k(t)\right) dt = 0.
		\]
		Since $\phi_k \in \Hs[0,\tau_k]$ is arbitrary,
		a standard property of test functions, 
		as stated for example in \cite[Proposition 13.2.2]{Tucsnak2009},
		implies that
		\[
		y_k^{(L)}(t) + \sum_{\ell=0}^{L-1}
		\left(
		R_{u,\ell} u_k^{(\ell)}(t) + R_{y,\ell} y_k^{(\ell)}(t)
		\right) - v_k(t) = 0
		\]
		for a.e.~$t \in [0,\tau_k]$. Thus, the differential equations given in \eqref{eq:system_k} hold.
	\end{proof}
	
	Lemma~\ref{lem:data_consistency_cond} shows that
	it is reasonable to quantify the noise intensity in terms of the
	operator $V_0$.
	Noting that $V_0V_0^*$ is
	a $p\times p$ matrix,
	we define the noise class $\Delta_{\tau,\Theta}$ by
	\[
	\Delta_{\tau,\Theta} \coloneqq \{
	(v_k)_{k=1}^{K}\in \Delta_{\tau}: V_0V_0^* \leq \Theta
	\},
	\]
	where the matrix
	$\Theta \in \mathbb{S}^p$ with $\Theta \geq 0$
	represents the noise intensity.
	We assume that the noise-intensity matrix
	$\Theta $ is known.
	The matrix $V_0V_0^*$ can be decomposed as follows:
	\begin{equation}
		\label{eq:V0V0*_decomposition}
		V_0V_0^* = \sum_{k=1}^K V_{0,k}V_{0,k}^*.
	\end{equation}
	It follows that 
	the noise signals associated with different input-output trajectories contribute independently without cross terms under the constraint 
	$V_0V_0^* \leq \Theta$.
	We briefly discuss this noise constraint 
	in the special case 
	$\Theta = c I_p$ with $c\geq 0$.
	\begin{remark}
		Observe first that
		$V_0V_0^* \leq c I_p$ if and only if 
		\[
		\sum_{k=1}^K \|V_{0,k}^* b\|_{\Hs}^2 \leq c \|b\|_{\mathbb{R}^p}^2 \quad 
		\text{for all $b \in \mathbb{R}^p$}.
		\]
		Moreover, 
		\[
		\sum_{k=1}^K \|V_{0,k}^* b\|_{\Hs}^2 \leq 
		\left( \sum_{k=1}^K \|V_{0,k}^* \|^2 \right) \|b\|_{\mathbb{R}^p}^2
		=
		\left( \sum_{k=1}^K \|V_{0,k} \|^2 \right) \|b\|_{\mathbb{R}^p}^2
		\]
		for all $b \in \mathbb{R}^p$.
		We now estimate 
		the operator norm $\|V_{0,k} \|$
		for a fixed $k=1,\dots,K$.
		Define the integral operator $\mathcal{J}_k \in 
		\mathcal{L}(\Ls([0,\tau_k];\mathbb{R}^p)) $ by
		\[
		(\mathcal{J}_k f)(t) \coloneqq \int_0^t f(s)ds,\quad 
		f \in\Ls([0,\tau_k];\mathbb{R}^p).
		\]
		Integrating by parts $L$ times, we have
		\[
		V_{0,k} \phi_k = (-1)^L \int_0^{\tau_k} \phi_k^{(L)}(t) 
		(\mathcal{J}_k^Lv_k)(t)dt
		\]
		for all $\phi_k \in \Hs[0,\tau_k]$.  
		This implies that $\|V_{0,k}\| \leq \|\mathcal{J}_k^Lv_k\|_{\Ls}$.
		Since repeated integration
		attenuates the high-frequency components of 
		$v_k$, 
		the constraint $V_0V_0^*\leq cI_p$ is
		easier to satisfy when 
		the noise signals $(v_k)_{k=1}^K$ have little low-frequency content.
		See \cite[Proposition~A]{Wakaiki2025Cont}
		for a frequency-domain representation of 
		$\|V_{0}\|$ in the
		case $L=1$ and $K=1$.
		\hspace*{\fill} $\triangle$ 
	\end{remark}

	For data $\mathfrak{D} = (u_k,y_k)_{k=1}^{K} \in \Gamma_{\tau}$ and 
	a noise-intensity matrix $\Theta \in \mathbb{S}^p$ with $\Theta \geq 0$, we define 
	the set $\Sigma_{\mathfrak{D},\Theta}$ of systems by
	\[
	\Sigma_{\mathfrak{D},\Theta} \coloneqq \{
	R \in \mathbb{R}^{p\times qL}:
	\text{there exists $(v_k)_{k=1}^{K} \in \Delta_{\tau,\Theta}$ such that 
		\eqref{eq:system_k} holds}\}.
	\]
	Systems in $\Sigma_{\mathfrak{D},\Theta}$ are 
	consistent with the data $\mathfrak{D}$
	corrupted by noise in the class  $\Delta_{\tau,\Theta}$.
	To provide a characterization of  $\Sigma_{\mathfrak{D},\Theta}$
	in terms of a matrix inequality,
	we define $\Pi\in \mathbb{S}^{p+qL}$
	by
	\begin{equation}
		\label{eq:Pi_def}
		\Pi \coloneqq 
		\begin{bmatrix}
			\Theta - Y_LY_L^* & -Y_L H^* \\
			-H Y_L^* & -H H^*
		\end{bmatrix},
	\end{equation}
	where 
	the operators $H$ and 
	$Y_L$ are 
	defined by \eqref{eq:H_def} and 
	\eqref{eq:YL_def}, respectively.
	The block matrices
	$\Theta- Y_LY_L^* 
	\in \mathbb{S}^{p}$,
	$Y_LH^* = (HY_L^*)^{\top}
	\in \mathbb{R}^{p \times qL}$, and $HH^*\in \mathbb{S}^{qL}$ of $\Pi$
	can be written as 
	\begin{align}
		\Theta- Y_LY_L^* &= \Theta -  \sum_{k=1}^{K}
		Y_{L,k} Y_{L,k}^*, \label{eq:I_YLYL_decompose} \\
		Y_LH^* &= \sum_{k=1}^{K}
		\begin{bmatrix}
			Y_{L,k}W_{0,k}^* & \cdots &
			Y_{L,k}W_{L-1,k}^*
		\end{bmatrix}, \label{eq:YH_decompose}\\
		HH^* &=  \sum_{k=1}^{K}
		\begin{bmatrix}
			W_{0,k}W_{0,k}^*  & \cdots &
			W_{0,k}W_{L-1,k}^* \\
			\vdots & \ddots &  \vdots \\
			W_{L-1,k}W_{0,k}^* & \cdots & W_{L-1,k}W_{L-1,k}^* 
		\end{bmatrix}.\label{eq:HH_decompose}
	\end{align}
	Matrices, such as
	$Y_{L,k} Y_{L,k}^*$, $Y_{L,k}W_{\ell,k}^*$, and 
	$W_{j,k}W_{\ell,k}^*$,
	can be computed from the integral representation
	presented in Proposition~\ref{prop:adjoint}.
	It is worth noting that the matrix $\Pi$ contains cross terms between 
	the derivatives, but not between different trajectories.
	\begin{lemma}
		\label{lem:system_rep}
		Let $\mathfrak{D} = (u_k,y_k)_{k=1}^{K} \in \Gamma_{\tau}$ and let $\Theta \in \mathbb{S}^p$ with $\Theta \geq 0$.
		Define  $\Pi\in \mathbb{S}^{p+qL}$
		by \eqref{eq:Pi_def}.
		Then
		the following statements are equivalent
		for 
		all $R \in \mathbb{R}^{p\times qL}$:
		\begin{enumerate}
			\renewcommand{\labelenumi}{\textup{(\roman{enumi})}}
			\item 
			$R \in \Sigma_{\mathfrak{D},\Theta}$.
			\item 
			$\begin{bmatrix}
				I_p \\ R^{\top} 
			\end{bmatrix}^{\top}
			\Pi
			\begin{bmatrix}
				I_p \\ R^{\top} 
			\end{bmatrix} \geq  0$.
		\end{enumerate} 
	\end{lemma}
	\begin{proof}
		Let 
		$R  \in \mathbb{R}^{p\times qL}$.
		By the definition \eqref{eq:Pi_def} of $\Pi$, we obtain
		\begin{equation}
			\label{eq:K_cRHY}
			\begin{bmatrix}
				I_p \\ R^{\top} 
			\end{bmatrix}^{\top}
			\Pi
			\begin{bmatrix}
				I_p \\ R^{\top} 
			\end{bmatrix} =
			\Theta  - (RH+Y_L) (RH+Y_L)^*.
		\end{equation}
		From this and Lemma~\ref{lem:data_consistency_cond}, 
		the implication (i) $\Rightarrow$ (ii) follows immediately.
		
		To prove the implication (ii) $\Rightarrow$ (i),
		first observe that
		\begin{equation}
			\label{eq:RHY_sum}
			(RH+Y_L) (RH+Y_L)^* =
			\sum_{k=1}^{K} (RH_k + Y_{L,k})(RH_k + Y_{L,k})^*,
		\end{equation}
		where the operator $H_k$ is defined by \eqref{eq:Hi_def}.
		By this and \eqref{eq:K_cRHY}, we obtain
		\begin{equation}
			\label{eq:RHY_sum_bound}
			\sum_{k=1}^{K} (RH_k + Y_{L,k})(RH_k + Y_{L,k})^*
			\leq  \Theta.
		\end{equation}
		The same calculation as in  \eqref{eq:RH}
		shows that
		\[
		(RH_k + Y_{L,k})\phi_k = 
		\int_0^{\tau_k} \phi_k(t)
		\left(y_k^{(L)}(t) + \sum_{\ell=0}^{L-1}
		\left(
		R_{u,\ell} u_k^{(\ell)}(t) + R_{y,\ell} y_k^{(\ell)}(t)
		\right) \right) dt
		\]
		for all $\phi_k \in \Hs[0,\tau_k]$ and $k=1,\dots,K$.
		It follows that 
		$RH_k+Y_{L,k}$ is the synthesis operator $V_{0,k}$ associated with
		\[
		v_k \coloneqq y_k^{(L)} + 
		\sum_{\ell=0}^{L-1}
		\left(
		R_{u,\ell} u_k^{(\ell)} + R_{y,\ell} y_k^{(\ell)}
		\right) \in \Ls([0,\tau_k];\mathbb{R}^p)
		\]
		for each $k=1,\dots,K$.
		Moreover,
		\eqref{eq:V0V0*_decomposition} and \eqref{eq:RHY_sum_bound} give
		$(v_k)_{k=1}^K \in \Delta_{\tau,\Theta}$.
		Since 
		\[
		(RH+Y_L)\phi  = \sum_{k=1}^K 
		(RH_k+Y_{L,k})\phi_k =  \sum_{k=1}^K V_{0,k}\phi_k  = V_0 \phi
		\]
		for all $\phi = (\phi_k)_{k=1}^K \in \Ht$,
		Lemma~\ref{lem:data_consistency_cond}
		shows that 
		statement~(i) holds.
	\end{proof}
	
	\begin{remark}
		The role of continuous-time data in the proposed method 
		is to compute
		the matrix $\Pi$ through 
		the integral representation in
		Proposition~\ref{prop:adjoint}.
		Hence, our framework is not restricted to the
		ideal case where
		continuous-time data are 
		directly available.
		Even when measurements 
		are obtained only at sampling instants, 
		the
		proposed method remains practically
		implementable, 
		provided that 
		the sampling period is
		small enough for the integrals to be
		approximated accurately by numerical integration.
		\hspace*{\fill} $\triangle$ 
	\end{remark}

	\section{Data-driven stabilization}
	\label{sec:stabilization}
	In this section, we state
	our main result on data-driven stabilization
	of continuous-time systems.
	We begin in
	Section~\ref{sec:closed_loop_stability}
	by presenting a necessary
	and sufficient condition for model-based 
	stabilization within the behavioral framework.
	Using this condition,
	we 
	introduce and characterize the concept of 
	data informativity for quadratic stabilization
	in Section~\ref{sec:informativity_QS}.
	\subsection{Closed-loop system}
	\label{sec:closed_loop_stability}
	For feedback stabilization of 
	the system~\eqref{eq:system_original},
	we consider the controller
	\begin{equation}
		\label{eq:controller}
		u^{(L)}(t) + \sum_{\ell=0}^{L-1}
		\left(
		C_{u,\ell} u^{(\ell)}(t)  + C_{y,\ell} y^{(\ell)}(t) 
		\right)
		= 0,\quad t \geq 0,
	\end{equation}
	where $C_{u,\ell} \in \mathbb{R}^{m \times m}$
	and $C_{y,\ell} \in \mathbb{R}^{m\times p}$
	for $\ell=0,\dots,L-1$.
	At time $t$,
	this controller receives $y(t)$ from the system \eqref{eq:system_original}
	and then applies $u(t)$ to it.
	Define the matrix $C \in  \mathbb{R}^{m\times qL}$
	by
	\[
	C \coloneqq 
	\begin{bmatrix}
		C_{u,0} & C_{y,0} & C_{u,1} & C_{y,1} & \cdots & C_{u,L-1} & C_{y,L-1}
	\end{bmatrix}.
	\]
	As in the case of the system \eqref{eq:system_original}, we 
	identify the matrix $C$ with the controller \eqref{eq:controller}.
	We define the $q \times q$ polynomial matrix $P$ by
	\[
	P(\xi) \coloneqq 
	I_{q} \xi^L + 
	\sum_{\ell=0}^{L-1}
	\begin{bmatrix}
		C_{u,\ell} & C_{y,\ell} \\
		R_{u,\ell} & R_{y,\ell}
	\end{bmatrix} \xi^{\ell}.
	\]
	In the noise-free case,
	the closed-loop system 
	obtained by interconnecting the system  \eqref{eq:system_original} and 
	the controller \eqref{eq:controller}
	can be written as 
	\begin{equation}
		\label{eq:closed_loop}
		P\left(
		\frac{d}{dt}
		\right)
		w =0,\quad 
		\text{where~}
		w \coloneqq 
		\begin{bmatrix}
			u \\ y
		\end{bmatrix}.
	\end{equation}
	We define the behavior $\mathcal{B}$ of the closed-loop 
	system by
	\begin{equation}
		\label{eq:closed_loop_behavior}
		\mathcal{B} \coloneqq 
		\left\{
		w \in \Cinf(\mathbb{R};\mathbb{R}^q):
		P\left(
		\frac{d}{dt}
		\right)
		w =0
		\right\}.
	\end{equation}

	\begin{definition}
		We say that 
		the controller $C \in \mathbb{R}^{m\times qL}$ {\em stabilizes} the system $R \in \mathbb{R}^{p\times qL}$ 
		if $\lim_{t\to \infty} w(t) = 0$
		for all $w \in \mathcal{B}$.
	\end{definition}

	Necessary and sufficient LMI conditions 
	for stability were
	derived from the behavioral
	viewpoint
	in \cite{Kaneko1998,Cotroneo2000}.
	Here we present a simpler condition, using
	the property that
	the leading coefficient matrix of $P$ is the identity matrix.
	Define $J_{q(L-1)} \in \mathbb{R}^{q(L-1)\times qL}$ by
	\begin{equation}
		\label{eq:J_def}
		J_{q(L-1)} \coloneqq 
		\begin{bmatrix}
			0_{q(L-1) \times q} & I_{q(L-1)}
		\end{bmatrix}.
	\end{equation}
	Since the proof of the following lemma 
	relies on behavioral theory and
	quadratic differential forms, we defer it, together with the required
	background, to Appendix~\ref{sec:Proof_of_Lemma}.
	\begin{lemma}
		\label{lem:QMI}
		The controller $C\in \mathbb{R}^{m\times qL}$ stabilizes the system $R\in  \mathbb{R}^{p\times qL}$  if and only if
		there exists a matrix
		$\Psi \in \mathbb{S}^{qL}$ such that
		$\Psi > 0$ and 
		\begin{equation}
			\label{eq:LMI_stabilization}
			\begin{bmatrix}
				-J_{q(L-1)} \\ C \\ R
			\end{bmatrix}^{\top}
			\Psi + \Psi 
			\begin{bmatrix}
				-J_{q(L-1)} \\ C \\ R
			\end{bmatrix} > 0,
		\end{equation}
		where the matrix 
		$J_{q(L-1)}$ is
		as in \eqref{eq:J_def}. 
	\end{lemma}

	\subsection{Informativity for quadratic stabilization}
	\label{sec:informativity_QS}
	We now introduce the notion of 
	data informativity for quadratic stabilization 
	based on 
	the necessary and sufficient condition for stabilization given in Lemma~\ref{lem:QMI}.
	The following definition 
	is a continuous-time analogue of 
	\cite[Definition~10]{Waarde2023TAC}.
	\begin{definition}
		Given a noise-intensity matrix
		$\Theta \in \mathbb{S}^p$ with
		$\Theta \geq 0$,
		the data $\mathfrak{D} = (u_k,y_k)_{k=1}^{K} \in \Gamma_{\tau}$
		are called {\em informative for quadratic stabilization 
			under the noise class $\Delta_{\tau, \Theta}$} if
		there exist matrices $\Psi \in \mathbb{S}^{qL}$ and 
		$C \in \mathbb{R}^{m\times qL}$ such that 
		$\Psi > 0$ and
		the inequality \eqref{eq:LMI_stabilization} holds
		for all $R \in \Sigma_{\mathfrak{D},\Theta}$.
	\end{definition}

	Let $\mathcal{M},\mathcal{N}  \in 
	\mathbb{S}^{q+r}$  be
	partitioned as 
	\begin{equation}
		\label{eq:MN_partition}
		\mathcal{M} =
		\begin{bmatrix}
			\mathcal{M}_{11} & \mathcal{M}_{12} \\
			\mathcal{M}_{12}^{\top} & \mathcal{M}_{22}
		\end{bmatrix}\quad \text{and}
		\quad 
		\mathcal{N} =
		\begin{bmatrix}
			\mathcal{N}_{11} & \mathcal{N}_{12} \\
			\mathcal{N}_{12}^{\top} & \mathcal{N}_{22}
		\end{bmatrix},
	\end{equation}
	where 
	the $(1,1)$-blocks have size $q \times q$ and 
	the $(2,2)$-blocks have size $r \times r$.
	Define
	the matrix sets $\mathcal{Z}_{q,r}(\mathcal{N})$
	and $\mathcal{Z}_{q,r}^+(\mathcal{M})$
	by
	\begin{align}
		\label{eq:ZN_def}
		\mathcal{Z}_{q,r}(\mathcal{N}) &\coloneqq 
		\left\{
		Z \in \mathbb{R}^{r\times q} :
		\begin{bmatrix}
			I_q \\ Z
		\end{bmatrix}^{\top}
		\mathcal{N}
		\begin{bmatrix}
			I_q \\Z
		\end{bmatrix} \geq  0
		\right\}, \\
		\label{eq:ZM_def}
		\mathcal{Z}_{q,r}^+(\mathcal{M}) &\coloneqq 
		\left\{
		Z \in \mathbb{R}^{r\times q} :
		\begin{bmatrix}
			I_q \\ Z
		\end{bmatrix}^{\top}
		\mathcal{M}
		\begin{bmatrix}
			I_q \\ Z
		\end{bmatrix}> 0
		\right\}.
	\end{align}
	The following result is called the {\em strict matrix $S$-lemma}
	and 
	is is useful for studying data informativity.
	See \cite[Theorem~4.10]{Waarde2023SIAM} for the proof.
	\begin{lemma}
		\label{lem:S_lemma}
		Let $\mathcal{M},\mathcal{N} \in 
		\mathbb{S}^{q+r}$ be partitioned as in \eqref{eq:MN_partition}. Assume that $\mathcal{N}$ satisfies
		\begin{equation}
			\label{eq:S_Lemma_cond}
			\mathcal{N}_{22} < 0,\quad 
			\Ker  \mathcal{N}_{22} \subseteq
			\Ker  \mathcal{N}_{12},\quad \text{and}
			\quad 
			\mathcal{N}_{11} - 
			\mathcal{N}_{12}\mathcal{N}_{22}^{-1}\mathcal{N}_{12}^{\top}  \geq  0.
		\end{equation}
		Then $\mathcal{Z}_{q,r}(\mathcal{N}) \subseteq
		\mathcal{Z}_{q,r}^{+}(\mathcal{M})$
		if and only if 
		there exists a scalar $\alpha \geq 0$
		such that
		$\mathcal{M} - \alpha \mathcal{N} > 0$.
	\end{lemma}

	We assume that
	the input-output data $(u_k,y_k)_{k=1}^{K} \in \Gamma_{\tau}$ are generated by
	an unknown true system in the presence of noise in the class
	$\Delta_{\tau, \Theta}$, where the
	noise-intensity matrix
	$\Theta \in \mathbb{S}^p$ with
	$\Theta \geq 0$ is known.
	\begin{assumption}
		\label{assump:true_system}
		There exist a system
		$R_s \in \mathbb{R}^{p\times qL}$
		and a noise sequence $(v_{s,k})_{k=1}^{K}\in \Delta_{\tau, \Theta}$ 
		such that
		\[
		y_k^{(L)}(t) + R_s
		\begin{bmatrix}
			u_k(t) \\ y_k(t)  \\ \vdots \\ u_k^{(L-1)}(t) \\ y_k^{(L-1)}(t)
		\end{bmatrix} = v_{s,k}(t)
		\]
		holds for a.e.~$t \in[0,\tau_k]$ and all $k=1,\dots, K$.
	\end{assumption}

	We make an assumption on the data-embedded 
	operator $H$ introduced in Section~\ref{sec:characterization_data_consistency}. 
	Recall that 
	$H$ is a continuous-time counterpart of
	a data Hankel matrix used in discrete-time data-driven control.
	\begin{assumption}
		\label{assump:surjectivity}
		The matrix $HH^* \in \mathbb{R}^{qL\times qL}$
		given in \eqref{eq:HH_decompose} is invertible.
	\end{assumption}
	
	We can numerically 
	check whether Assumption~\ref{assump:surjectivity}
	is satisfied, using Proposition~\ref{prop:adjoint} and
	\eqref{eq:HH_decompose}.
	In Section~\ref{sec:surjectivity},
	we will see that in the noise-free setting,
	this invertibility property
	is equivalent to data informativity for system identification.
	
	We define
	$J_{p}\in \mathbb{R}^{p \times qL}$
	and 
	$\mathcal{N}\in \mathbb{S}^{2qL}$
	by 
	\begin{align}
		\label{eq:Jp_def} 
		J_p &\coloneqq 
		\begin{bmatrix}
			0_{p\times (qL-p)} & I_p
		\end{bmatrix}, \\
		\label{eq:calN_def} 
		\mathcal{N} &\coloneqq
		\begin{bmatrix}
			J_p & 0_{p \times qL}  \\
			0_{qL \times qL} & I_{qL}
		\end{bmatrix}^{\top}
		\Pi
		\begin{bmatrix}
			J_p & 0_{p \times qL} \\
			0_{qL \times qL} & I_{qL}
		\end{bmatrix},
	\end{align}
	where the matrix $\Pi $
	is as in \eqref{eq:Pi_def}.
	To apply Lemma~\ref{lem:S_lemma}, we have to 
	prove that the matrix
	$\mathcal{N}$ satisfies
	the conditions given in
	\eqref{eq:S_Lemma_cond}.
	\begin{lemma}
		\label{lem:cond_check_for_Slemma}
		Suppose that Assumptions~\ref{assump:true_system}
		and \ref{assump:surjectivity} hold for 
		the data $\mathfrak{D} = (u_k,y_k)_{k=1}^{K}
		\in \Gamma_{\tau}$
		and
		the noise-intensity matrix
		$\Theta \in \mathbb{S}^p$ with
		$\Theta\geq 0$. Define
		$\mathcal{N}\in \mathbb{S}^{2qL}$ 
		by \eqref{eq:calN_def} and partition it as in
		\eqref{eq:MN_partition}, where the diagonal blocks
		$\mathcal{N}_{11}$ and $\mathcal{N}_{22}$
		have size $qL\times qL$.
		Then $\mathcal{N}$ satisfies the conditions
		given in \eqref{eq:S_Lemma_cond}.
	\end{lemma}
	\begin{proof}
		Since $\mathcal{N}_{22} = 
		-H H^*$, 
		it follows that $\mathcal{N}_{22} < 0$ by 
		Assumption~\ref{assump:surjectivity}.
		This also yields
		\[
		\Ker  \mathcal{N}_{22} = \{ 0\} \subseteq
		\Ker  \mathcal{N}_{12}.
		\]
		If we define 
		\[
		\Pi_{11} \coloneqq  \Theta - Y_LY_L^*,\quad 
		\Pi_{12} \coloneqq -Y_LH^*,\quad \text{and} \quad  
		\Pi_{22} \coloneqq -HH^*, 
		\]
		then
		\begin{equation}
			\label{eq:N_Pi_relation22}
			\mathcal{N}_{11} - 
			\mathcal{N}_{12}\mathcal{N}_{22}^{-1}\mathcal{N}_{12}^{\top} =
			\begin{bmatrix}
				0_{(qL-p) \times (qL-p) }  & 0_{(qL-p)\times p} \\
				0_{p \times (qL-p)} & \Pi_{11} - \Pi_{12}
				\Pi_{22}^{-1} \Pi_{12}^{\top}
			\end{bmatrix}.
		\end{equation}
		By \cite[Fact 8.9.7]{Bernstein2018},
		\begin{equation}
			\label{eq:Pi22_rep}
			\Pi_{11} - 
			\Pi_{12}\Pi_{22}^{-1}\Pi_{12}^{\top}  
			=
			\begin{bmatrix}
				I_p \\ Z
			\end{bmatrix}^{\top}
			\Pi
			\begin{bmatrix}
				I_p \\ Z
			\end{bmatrix} -
			(Z+\Pi_{22}^{-1}\Pi_{12}^{\top})^{\top} 
			\Pi_{22} (Z+\Pi_{22}^{-1}\Pi_{12}^{\top})
		\end{equation}
		for all $Z \in \mathbb{R}^{qL\times p}$.
		Using Assumption~\ref{assump:true_system} and Lemma~\ref{lem:system_rep}, we obtain
		\begin{equation}
			\label{eq:Rs_Pi}
			\begin{bmatrix}
				I_p \\ R_s^{\top} 
			\end{bmatrix}^{\top}
			\Pi
			\begin{bmatrix}
				I_p \\ R_s^{\top} 
			\end{bmatrix} \geq 
			0.
		\end{equation}
		From 
		\eqref{eq:N_Pi_relation22}--\eqref{eq:Rs_Pi},
		we conclude that
		$\mathcal{N}_{11} - 
		\mathcal{N}_{12}\mathcal{N}_{22}^{-1}\mathcal{N}_{12}^{\top}
		\geq  0$.
	\end{proof}
	
	We are now ready to characterize 
	data informativity for quadratic stabilization.
	The following theorem is a continuous-time counterpart to
	\cite[Theorem~20]{Waarde2023TAC}.
	\begin{theorem}
		\label{thm:QS}
		Suppose that Assumptions~\ref{assump:true_system}
		and \ref{assump:surjectivity} hold for 
		the data $\mathfrak{D} = (u_k,y_k)_{k=1}^{K}
		\in \Gamma_{\tau}$
		and
		the noise-intensity matrix
		$\Theta \in \mathbb{S}^p$ with
		$\Theta\geq 0$.
		Then
		the following statements are equivalent:
		\begin{enumerate}
			\renewcommand{\labelenumi}{\textup{(\roman{enumi})}}
			\item 
			The data $\mathfrak{D}$
			are informative for quadratic stabilization 
			under the noise class $\Delta_{\tau, \Theta}$.
			\item 
			\label{enum:LMI}
			There exist matrices $\Phi \in \mathbb{S}^{qL}$
			and 
			$D \in \mathbb{R}^{m \times qL}$ such that 
			$\Phi > 0$ and
			\begin{equation}
				\label{eq:LMI_stabilization_thm}
				\begin{bmatrix}
					\begin{bmatrix}
						-J_{q(L-1)} \Phi  \\ D \\ 0_{p\times qL}
					\end{bmatrix}^{\top}  + 
					\begin{bmatrix}
						-J_{q(L-1)}\Phi \\ D \\  0_{p\times qL}
					\end{bmatrix} & \Phi \\
					\Phi & 0_{qL \times qL}
				\end{bmatrix} -
				\mathcal{N} > 0,
			\end{equation}
			where the matrices 
			$J_{q(L-1)}$ 
			and
			$\mathcal{N}$ 
			are defined by \eqref{eq:J_def} and
			\eqref{eq:calN_def},
			respectively. 
		\end{enumerate}
		Moreover, if statement~\textup{(ii)} holds, then
		the controller
		$C \coloneqq D \Phi^{-1}$ stabilizes all systems $R \in 
		\Sigma_{\mathfrak{D},\Theta}$.
	\end{theorem}

	\begin{proof}
		First, suppose that statement~(i) holds.
		Then there exist $\Psi \in \mathbb{S}^{qL}$ and 
		$C \in \mathbb{R}^{m\times qL}$
		such that
		$\Psi > 0$ and 
		the inequality \eqref{eq:LMI_stabilization} holds
		for all $R \in \Sigma_{\mathfrak{D},\Theta}$.
		Define
		$\mathcal{M}  \in \mathbb{S}^{2qL}$ by
		\begin{equation}
			\label{eq:M_def}
			\mathcal{M} \coloneqq 
			\begin{bmatrix}
				\Psi^{-1} \begin{bmatrix}
					-J_{q(L-1)} \\ C \\  0_{p\times qL}
				\end{bmatrix}^{\top} + 
				\begin{bmatrix}
					-J_{q(L-1)} \\ C \\  0_{p\times qL}
				\end{bmatrix}\Psi^{-1}  & \Psi^{-1} \\
				\Psi^{-1} & 0_{qL \times qL}
			\end{bmatrix}.
		\end{equation}
		A routine calculation shows that 
		for all $R \in \mathbb{R}^{p \times qL}$,
		\begin{equation}
			\label{eq:Lyap_QMI}
			\Psi^{-1}
			\begin{bmatrix}
				-J_{q(L-1)} \\ C \\ R
			\end{bmatrix}^{\top}
			+
			\begin{bmatrix}
				-J_{q(L-1)} \\ C \\ R
			\end{bmatrix} \Psi^{-1} 
			=
			\begin{bmatrix}
				I_{qL} \\
				R^{\top} 
				J_p
			\end{bmatrix}^{\top}
			\mathcal{M} 
			\begin{bmatrix}
				I_{qL} \\
				R^{\top} 
				J_p
			\end{bmatrix},
		\end{equation}
		where the matrix $J_p$ is defined by \eqref{eq:Jp_def}.
		By \eqref{eq:LMI_stabilization} and 
		\eqref{eq:Lyap_QMI}, we obtain
		the following implication:
		\begin{equation}
			\label{eq:NM_relation}
			R \in \Sigma_{\mathfrak{D},\Theta}
			\quad \Rightarrow \quad 
			R^{\top} J_p\in \mathcal{Z}_{qL,qL}^{+}(\mathcal{M}).
		\end{equation}
		From Lemma~\ref{lem:system_rep},
		it follows that 
		$R \in \Sigma_{\mathfrak{D},\Theta}$ if and only if
		$R^{\top} \in \mathcal{Z}_{p,qL}(\Pi)$, where 
		the matrix $\Pi$ is as in \eqref{eq:Pi_def}.
		Since the $(2,2)$-block $\Pi_{22} = -HH^*$ of $\Pi$
		is invertible by Assumption~\ref{assump:surjectivity},
		we deduce from 
		\cite[Theorem~3.4]{Waarde2023SIAM} that
		\begin{equation}
			\label{eq:N_tilde_N_relation}
			\left\{
			R^{\top} J_p:
			R \in \Sigma_{\mathfrak{D},\Theta}
			\right\} =
			\mathcal{Z}_{qL,qL}(\mathcal{N}).
		\end{equation}
		It follows from
		\eqref{eq:NM_relation}
		and 
		\eqref{eq:N_tilde_N_relation} that
		\[
		\mathcal{Z}_{qL,qL}(\mathcal{N}) \subseteq
		\mathcal{Z}_{qL,qL}^{+}(\mathcal{M}).
		\]
		By Lemmas~\ref{lem:S_lemma}
		and \ref{lem:cond_check_for_Slemma},
		there exists a scalar $\alpha \geq 0$ such that 
		$\mathcal{M} - \alpha \mathcal{N} > 0$.
		Since the $(2,2)$-block of $\mathcal{M}$
		is the zero matrix $0_{qL\times qL}$, we obtain $\alpha >0$.
		If we define
		$D \coloneqq (1/\alpha)C \Psi^{-1} \in \mathbb{R}^{m\times qL}$ and $\Phi \coloneqq (1/\alpha)\Psi^{-1} \in \mathbb{S}^{qL}$, then
		$\mathcal{M} - \alpha \mathcal{N} > 0$ can be written as
		the LMI \eqref{eq:LMI_stabilization_thm}.
		Thus, statement~(ii) holds.
		
		Conversely, suppose that statement~(ii) holds.
		Define $C \coloneqq D\Phi^{-1}$ and $\Psi \coloneqq \Phi^{-1}$. Then the LMI \eqref{eq:LMI_stabilization_thm} is equivalent to
		$\mathcal{M} - \mathcal{N} > 0$,
		where the matrix 
		$\mathcal{M} $
		is as in \eqref{eq:M_def}.
		Since \eqref{eq:Lyap_QMI} and \eqref{eq:N_tilde_N_relation} remain valid without
		assuming that statement~(i) holds,
		it follows that
		for all $R \in \Sigma_{\mathfrak{D},\Theta}$,
		\begin{align*}
			0 &<
			\begin{bmatrix}
				I_{qL} \\
				R^{\top} 
				J_p
			\end{bmatrix}^{\top}
			(\mathcal{M} -  \mathcal{N})
			\begin{bmatrix}
				I_{qL} \\
				R^{\top} 
				J_p
			\end{bmatrix} \leq 
			\Psi^{-1}
			\begin{bmatrix}
				-J_{q(L-1)} \\ C \\ R
			\end{bmatrix}^{\top}
			+
			\begin{bmatrix}
				-J_{q(L-1)} \\ C \\ R
			\end{bmatrix} \Psi^{-1}.
		\end{align*}
		Thus, statement~(i) holds, and Lemma~\ref{lem:QMI} shows that
		$C$ stabilizes all systems $R \in 
		\Sigma_{\mathfrak{D},\Theta}$.
	\end{proof}
	
	\section{Informativity of noise-free data
		for system identification}
	\label{sec:surjectivity}
	In this short section, we consider noise-free data and 
	characterize informativity
	for system identification in terms of the data-embedded operator $H$
	defined by \eqref{eq:H_def}.
	As in the case of quadratic stabilization, system identification reduces to matrix computations.
	
	For
	input-output data $(u_k,y_k)_{k=1}^{K} \in \Gamma_{\tau}$,
	we consider the noise-free system 
	\begin{equation}
		\label{eq:system_k_noise_free}
		y_k^{(L)}(t) + R
		\begin{bmatrix}
			u_k(t) \\ y_k(t)  \\ \vdots \\ u_k^{(L-1)}(t) \\ y_k^{(L-1)}(t)
		\end{bmatrix} = 0\quad \text{for a.e.~$t \in[0,\tau_k]$
			and all $k=1,\dots, K$}.
	\end{equation}
	We say that $(u_k,y_k)_{k=1}^{K}$ are {\em informative 
		for system identification} if
	\eqref{eq:system_k_noise_free} is satisfied for at most one system
	$R\in \mathbb{R}^{p\times qL}$.
	The following result is an extension of the characterization in
	the state-feedback case~\cite[Proposition~3.2]{Wakaiki2025Cont}.
	Recall that
	the matrices $Y_LH^* $ and $HH^*$ can be computed by combining Proposition~\ref{prop:adjoint}
	with \eqref{eq:YH_decompose} and
	\eqref{eq:HH_decompose}, respectively.
	\begin{proposition}
		\label{prop:identification}
		Assume that for the input-output data $(u_k,y_k)_{k=1}^{K}
		\in \Gamma_{\tau}$,
		there exists a system 
		$R_s \in \mathbb{R}^{p\times qL}$  such that 
		\eqref{eq:system_k_noise_free} with $R=R_s$ holds.
		Define the operators $H$ and $Y_L$ 
		by 
		\eqref{eq:H_def} and \eqref{eq:YL_def}, respectively.
		Then
		the following statements are equivalent:
		\begin{enumerate}
			\renewcommand{\labelenumi}{\textup{(\roman{enumi})}}
			\item The data $(u_k,y_k)_{k=1}^{K}$
			are informative for system identification.
			\item $H$ is surjective.
			\item $HH^* \in \mathbb{S}^{qL}$ is invertible. 
		\end{enumerate}
		Moreover, if statement~\textup{(iii)} holds, 
		then 
		\begin{equation}
			\label{eq:Rs_rep}
			R_s = -Y_L H^* (HH^*)^{-1}.
		\end{equation}
	\end{proposition}

	\begin{proof}
		Let $R\in \mathbb{R}^{p\times qL}$.
		Lemma~\ref{lem:data_consistency_cond}
		shows that 
		the differential equations given in \eqref{eq:system_k_noise_free} hold if and only if
		the operator equation
		\begin{equation}
			\label{eq:noise_free_op_eq}
			RH + Y_L = 0
		\end{equation}
		holds.
		Therefore,
		the equivalence of (i) and (ii) can be proved
		in the same way as in the state-feedback case
		\cite[Proposition~3.2]{Wakaiki2025Cont}.
		The equivalence of (ii) and (iii)
		is a well-known fact for bounded linear 
		operators; see, e.g., \cite[Proposition~12.1.3]{Tucsnak2009}.
		Finally, if statement~(iii) holds, then
		$H H^* (HH^*)^{-1} = I_{qL}$.
		From this and \eqref{eq:noise_free_op_eq},
		we conclude that
		\eqref{eq:Rs_rep} holds.
	\end{proof}
	
	In Appendix~\ref{sec:approximation},
	we provide another characterization of
	informativity for system identification in terms of
	B-splines.
	
	\section{Example}
	\label{sec:example}
	We consider an inverted pendulum mounted on 
	a cart.
	Let $x(t)$ be the horizontal position of the cart 
	at time $t$, and let $\theta(t)$ be the angle 
	of the pendulum measured 
	from the upright
	equilibrium
	at time $t$. 
	The control input $u(t)$ is the horizontal 
	force applied to the cart.
	The system parameters are defined as follows:
	$M_{\mathrm{c}}$ and $M_{\mathrm{p}}$
	are the masses of the cart and the pendulum, respectively;
	$\mu_{\mathrm{c}}$
	and $\mu_{\mathrm{p}}$ are 
	the viscous friction coefficients for the cart and
	the pendulum, respectively;
	$r$ is the distance from the 
	pivot to the center of mass of the pendulum;
	$J_{\mathrm{m}}$ is the pendulum's 
	moment of inertia about the center of mass; and 
	$g$ is the 
	gravitational acceleration.
	Linearizing the nonlinear equations of motion around the upright
	equilibrium yields
	\begin{equation*}
		\left\{
		\begin{aligned}
			&(M_{\mathrm{c}} + M_{\mathrm{p}}) x''(t) + \mu_{\mathrm{c}} x'(t) + M_{\mathrm{p}}r  \theta''(t)
			= u(t), \\
			& M_{\mathrm{p}} r x''(t) + (J_{\mathrm{m}} + 
			M_{\mathrm{p}}r^2) \theta''(t) + 
			\mu_{\mathrm{p}} \theta'(t) - M_{\mathrm{p}} g r \theta(t) = 0
		\end{aligned}
		\right.
	\end{equation*}
	for $t \geq 0$.
	Let $y(t) \coloneqq \begin{bmatrix}
		x(t) & \theta(t)
	\end{bmatrix}^{\top}$ be
	the output at time $t$.
	Choosing $R_{y,1},R_{y,0}
	\in \mathbb{R}^{2\times 2}$ and 
	$R_{u,0} \in \mathbb{R}^{2\times 1}$ appropriately,
	this system can be written as
	\[
	y''(t) + R_{y,1}y'(t) + R_{y,0}y(t) + R_{u,0}u(t) = 0,\quad 
	t \geq 0.
	\]
	As in the general form \eqref{eq:system_original},
	we incorporate the noise $v$ into the above system.
	The true parameter values used for data generation 
	are set as follows:
	\begin{gather*}
		M_{\mathrm{c}}= 0.6\,[\mathrm{kg}],\quad M_{\mathrm{p}} = 0.3\,[\mathrm{kg}],\quad 
		\mu_{\mathrm{c}} = 0.1\,[\mathrm{kg}/\mathrm{s}],  \quad 
		\mu_{\mathrm{p}} = 0.0003\,[\mathrm{kg}\cdot 
		\mathrm{m}^2/\mathrm{s}],\\
		r = 0.5\,[\mathrm{m}],\quad  J_{\mathrm{m}} = 0.025\,[\mathrm{kg}\cdot 
		\mathrm{m}^2],\quad 
		g = 9.81\,[\mathrm{m}/
		\mathrm{s}^2].
	\end{gather*}
	
	With $\tau = 0.5$, we  generate input-output data
	on the common interval $[0,\tau]$ 
	and discard any trajectory satisfying $|\theta(t)| > 0.1$
	for some $t \in [0,\tau]$ 
	to maintain the validity of the linear approximation.
	We repeat this process until we collect $25$ valid trajectories.
	The initial cart position $x(0)$ and 
	velocity $x'(0)$ 
	are sampled uniformly from the interval $(-0.1,0.1)$,
	and
	the initial angle $\theta(0)$ and angular velocity
	$\theta'(0)$ are sampled uniformly
	from the interval $(-0.01,0.01)$.
	The input $u$ is generated as
	\[
	u(t) =\sum_{i=1}^5 a_i \sin(2\pi f_i t + \varphi_i),\quad t \geq 0,
	\]
	where
	the amplitudes $a_i$, the
	frequencies $f_i$, and the initial phases $\varphi_i$
	are independently sampled from
	the uniform distributions on the intervals $(0,0.2)$, 
	$(0,5)$, and $(0,2\pi)$, respectively.
	Zero-mean Gaussian white noise
	$v$ with covariance matrix
	\[
	\mathrm{E}[v(t)v(s)^{\top}] = 
	2 \times 10^{-4}\delta(t-s)I_2
	\] is added, where 
	$\mathrm{E}[\cdot]$ is the expectation operator 
	and $\delta$ is the Dirac
	delta function.

	Let $\mathfrak{D} = (u_k,y_k)_{k=1}^{25}$
	be the valid data, and set the noise-intensity
	matrix to $\Theta= 
	10^{-6} I_2$.
	From the formula in Proposition~\ref{prop:adjoint}, we numerically verify that
	$HH^*$ is invertible, which implies that Assumption~\ref{assump:surjectivity} is 
	satisfied.
	Note that Assumption~\ref{assump:true_system} 
	is also satisfied.
	Indeed, applying 
	Proposition~\ref{prop:adjoint} 
	to
	the noise sequence 
	$(v_k)_{k=1}^{25}$ used to generate the data 
	$\mathfrak{D}$ gives
	$V_{0}V_{0}^* \leq 7.60 \times 10^{-7} I_2$, 
	which shows
	$(v_k)_{k=1}^{25} \in \Delta_{\tau,\Theta}$. 
	Therefore,
	the inverted pendulum with the 
	true parameter values belongs to the set
	$\Sigma_{\mathfrak{D},\Theta}$ of 
	data-consistent systems.
	
	By solving the LMIs in MATLAB
	R2026a with YALMIP~\cite{Lofberg2004}  and MOSEK~\cite{MOSEK2026},
	we find from Theorem~\ref{thm:QS}
	that the data
	$\mathfrak{D}$ are informative
	for quadratic stabilization under the noise class
	$\Delta_{\tau,\Theta}$.
	From a feasible solution of the LMIs,
	we can also compute a stabilizing controller
	\[
	u''(t) + 35u'(t) +
	804u(t) =
	701x'(t)
	+127 x(t) + 
	3600\theta'(t) + 
	15907 \theta(t),\quad t \geq 0.
	\]
	Since $\theta$ and $\theta'$
	have smaller magnitudes than $x$ and $x'$,
	their corresponding gains are larger.

	\section{Concluding remarks}
	\label{sec:conclusion}
	We have proposed a data-informativity
	approach for output feedback stabilization of
	continuous-time systems.
	The proposed synthesis-operator approach
	bridges the gap between
	continuous-time data and 
	discrete-time techniques.
	Employing behavioral theory, 
	we have established a necessary and sufficient condition
	under which noisy data are informative 
	for quadratic stabilization.
	Since the synthesis operators have finite-dimensional ranges,
	this condition can be formulated as LMIs.
	We have further explored data informativity
	for system identification in the noise-free setting.
	Future work includes extending 
	this synthesis-operator approach to input-output data corrupted by
	measurement noise.
	
	\appendix 
	\section{Proof of Lemma~\ref{lem:QMI}}
	\label{sec:Proof_of_Lemma}
	The proof of Lemma~\ref{lem:QMI}
	relies on 
	several results on quadratic differential forms 
	obtained in \cite{Willems1998}.
	We review these results in Section~\ref{sec:preliminaries_QDF}, and then prove
	Lemma~\ref{lem:QMI} in Section~\ref{sec:proof_of_Lyap}.
	
	\subsection{Background material on quadratic differential forms}
	\label{sec:preliminaries_QDF}
	Let $\Rsym$ denote
	the set of $q \times q$ real polynomial matrices $\Phi$ in the commuting indeterminates $\zeta$ and $\eta$ such that 
	$\Phi(\zeta,\eta) = \Phi(\eta,\zeta)^{\top}$.
	Let $\Phi_{i,j}\in\mathbb{R}^{q \times q}$ be 
	the coefficient 
	matrix corresponding to the monomial $\zeta^i\eta^j$ in $\Phi \in  \Rsym$,
	that is,
	\[
	\Phi(\zeta,\eta) = \sum_{i,j =0}^{d_0}
	\Phi_{i,j}\zeta^i\eta^j
	\]
	for some $d_0 \in \mathbb{N}_0$.
	The quadratic differential form $Q_{\Phi}\colon
	\Cinf(\mathbb{R};\mathbb{R}^q) \to
	\Cinf(\mathbb{R}) 
	$ induced by $\Phi$ is defined as 
	\[
	Q_{\Phi}(w)(t) \coloneqq \sum_{i,j =0}^{d_0} w^{(i)}(t)^{\top}
	\Phi_{i,j} w^{(j)}(t),\quad t \in \mathbb{R}.
	\]
	For $f \in \Cinf(\mathbb{R}) $,
	we write $f \geq 0$ if $f(t) \geq 0$
	for all $t \in \mathbb{R}$. Similarly,
	$\Phi \geq 0$ means that $Q_{\Phi}(w) \geq 0$ for all $w \in \Cinf(\mathbb{R};\mathbb{R}^q)$.
	
	Let $\Phi
	\in \Rsym \setminus \{ 0\}$. 
	Noting that $\Phi_{i,j} = \Phi_{j,i}^{\top}$,
	we define the degree $\deg(\Phi)$ of $\Phi$ by
	\[
	\deg(\Phi) \coloneqq \max
	\{i \in \mathbb{N}_0 :\Phi_{i,j}\ne 0 \text{~for some $j \in \mathbb{N}_0$}\}.
	\]
	Let $d \coloneqq \deg(\Phi)$.
	Using the coefficient matrices $\Phi_{i,j}$, 
	we define
	the matrix $\widetilde \Phi \in \mathbb{S}^{q(d+1)}$ by
	\[
	\widetilde \Phi \coloneqq 
	\begin{bmatrix}
		\Phi_{0,0} & \cdots & \Phi_{0,d} \\
		\vdots & \ddots & \vdots \\
		\Phi_{d,0} & \cdots & \Phi_{d,d}
	\end{bmatrix}.
	\]
	We call $\widetilde \Phi$ the {\em block coefficient matrix of $\Phi$}.
	For all $w \in \Cinf(\mathbb{R};\mathbb{R}^q) $
	and $t \in \mathbb{R}$, we have
	\[
	Q_{\Phi}(w)(t)  =
	\begin{bmatrix}
		w(t) \\ w'(t) \\ \vdots \\ w^{(d)}(t)
	\end{bmatrix}^{\top} \widetilde \Phi
	\begin{bmatrix}
		w(t) \\ w'(t) \\ \vdots \\ w^{(d)}(t)
	\end{bmatrix}.
	\]
	Given arbitrary $w_0,\dots,w_{d} \in \mathbb{R}^q$,
	we can construct 
	$w \in \Cinf(\mathbb{R};\mathbb{R}^q)$
	satisfying $w^{(\ell)}(0) = w_{\ell}$ for all $\ell=0,\dots,d$.
	Therefore,
	$\Phi \geq 0$ is equivalent to 
	$\widetilde \Phi \geq  0$.
	
	Consider a $q\times q$ real  polynomial matrix $P$
	of degree $L \in \mathbb{N}$ defined by
	\begin{equation}
		\label{eq:poly_matrix_G_def}
		P(\xi) \coloneqq 
		I_q \xi^{L} + P_{L-1}\xi^{L-1} + \cdots +P_0.
	\end{equation}
	Define the behavior $\mathcal{B}$ by
	\begin{equation}
		\label{eq:behavior}
		\mathcal{B} \coloneqq 
		\left\{
		w\in \Cinf(\mathbb{R};\mathbb{R}^q):
		P\left(\frac{d}{dt}\right) w = 0
		\right\}.
	\end{equation}
	The behavior $\mathcal{B}$ represents a dynamical system, 
	and hence we regard $\mathcal{B}$ as a system.
	We write \vspace{4pt}
	\begin{enumerate}
		\renewcommand{\labelenumi}{\textup{\alph{enumi})}}
		\item $\Phi \eB 0$ 
		if $Q_{\Phi}(w) = 0$ for all $w \in \mathcal{B}$;
		\vspace{4pt}
		\item $\Phi \geB 0$ 
		if $Q_{\Phi}(w) \geq 0$ for all $w \in \mathcal{B}$;
		\vspace{4pt}
		\item  $\Phi \gB 0$ 
		if $\Phi \geB 0$ and, for any $w \in \mathcal{B}$, $Q_{\Phi}(w) = 0$ implies $w = 0$; and \vspace{4pt}
		\item $\Phi \ggB 0$ 
		if $\Phi \geB 0$ and, for any $w \in \mathcal{B}$, $Q_{\Phi}(w)(0) = 0$ implies $w = 0$.
	\end{enumerate}
	\vspace{4pt}
	
	Let $\Phi \in \Rsym \setminus \{ 0\}$ and set $d \coloneqq \deg \Phi \leq L-1$. Since
	for any $w_0,\dots,w_d \in \mathbb{R}^q$,
	there exists $w \in \mathcal{B}$
	such that 
	$w^{(\ell)}(0) = w_{\ell}$ for all $\ell=0,\dots,d$,
	we obtain the following equivalences:
	\begin{align}
		\Phi \geB 0 \quad 
		&\Leftrightarrow \quad 
		\widetilde \Phi \geq 0, \label{eq:B_nn}\\
		\Phi \ggB 0 \quad 
		&\Leftrightarrow \quad 
		\widetilde \Phi >0 \hspace{6pt}\text{and} \hspace{6pt}
		\deg \Phi = L-1 .\label{eq:B_sp}
	\end{align}
	
	For
	$\Phi \in \Rsym $, we 
	define $\overset{\bullet}{\Phi}\in \Rsym$ by 
	\begin{equation}
		\label{eq:bullet_Phi}
		\overset{\bullet}{\Phi}(\zeta,\eta) \coloneqq 
		(\zeta+\eta) \Phi(\zeta,\eta).
	\end{equation}
	We say that
	the system $\mathcal{B}$ 
	is
	{\em stable} if $\lim_{t \to \infty} w(t) = 0$ for all $w \in \mathcal{B}$.
	In the theorem below,
	statement~a) is due to
	\cite[Proposition~4.7]{Willems1998}, and
	statement~b) follows from
	\cite[Theorem~4.12]{Willems1998}
	as a special case  with the following choice of
	$\Psi_0 \in \Rsym$:
	\begin{equation}
		\label{eq:Psi0}
		\Psi_0(\zeta,\eta) \coloneqq  -\sum_{i =0}^{L-1}
		I_q \zeta^i\eta^i.
	\end{equation}
	
	\begin{theorem}
		\label{thm:QDF}
		Let $P$ be a $q \times q$ real 
		polynomial matrix  of degree $L\in \mathbb{N}$ given by \eqref{eq:poly_matrix_G_def},
		and define $\mathcal{B}$
		by \eqref{eq:behavior}.
		Then the following statements hold:
		\begin{enumerate}
			\renewcommand{\labelenumi}{\textup{\alph{enumi})}}
			\item 	If 	there exists $\Psi \in \Rsym$
			such that $\Psi \geB 0$
			and $-\overset{\bullet}{\Psi} \gB  0$, then
			the system $\mathcal{B}$ is stable and $\Psi \ggB 0$.
			\item 
			If
			the system $\mathcal{B}$ is stable,
			then there exists $\Psi \in \Rsym$ with $\deg \Psi = L-1$ such that 
			$\Psi \ggB 0$ and 
			$\overset{\bullet}{\Psi} - \Psi_0 \eB 0$,
			where $\Psi_0$ is as in \eqref{eq:Psi0}.
		\end{enumerate}
	\end{theorem}
	
	\subsection{A necessary and 
		sufficient LMI condition}
	\label{sec:proof_of_Lyap}
	The following lemma presents
	a matrix inequality that is sufficient for
	$-\overset{\bullet}{\Psi} \gB  0$
	and is necessary for 
	$\overset{\bullet}{\Psi} - \Psi_0 \eB 0$.
	\begin{lemma}
		\label{lem:Lyap_ineq} 
		Let $P$ be a real $q \times q$
		polynomial matrix  of degree $L \in \mathbb{N}$ given by \eqref{eq:poly_matrix_G_def}, and 
		define $\widetilde P \in \mathbb{R}^{q \times qL}$ by
		\begin{align*}
			\widetilde P &\coloneqq 
			\begin{bmatrix}
				P_0 & P_1 & \cdots & P_{L-1}
			\end{bmatrix}.
		\end{align*}
		Let
		$\Psi \in \Rsym$ with $\deg(\Psi) = L-1$, and 
		let 
		$\widetilde \Psi \in \mathbb{S}^{qL}$ be
		the block coefficient matrix of $\Psi$.
		Then the following statements hold for
		the system $\mathcal{B}$ defined 
		by \eqref{eq:behavior}:
		\begin{enumerate}
			\renewcommand{\labelenumi}{\textup{\alph{enumi})}}
			\item
			Define $J_{q(L-1)} \coloneqq \begin{bmatrix}
				0_{q(L-1) \times q}  & I_{q(L-1)}
			\end{bmatrix}$.
			If the matrices $\widetilde P$ and $\widetilde \Psi$ satisfy 
			\begin{equation}
				\label{eq:JRPsi_ineq1}
				\begin{bmatrix}
					-J_{q(L-1)} \vspace{3pt}\\ \widetilde P
				\end{bmatrix}^{\top} \widetilde \Psi +  \widetilde \Psi 
				\begin{bmatrix}
					-J_{q(L-1)} \vspace{3pt}\\ \widetilde P
				\end{bmatrix} > 0,
			\end{equation}
			then $-\overset{\bullet}{\Psi} \gB  0$.
			\item 
			Let $\Psi_0 \in \Rsym$ be defined as in
			\eqref{eq:Psi0}.
			If $\overset{\bullet}{\Psi} - \Psi_0 \eB 0$, then
			\eqref{eq:JRPsi_ineq1} holds.
		\end{enumerate}
	\end{lemma}
	\begin{proof}
		Let $w \in \mathcal{B}$. 
		Before proving
		statements~a) and b), we first obtain
		a representation of $Q_{\overset{\bullet}{\Psi}}(w)$
		in terms of the matrices 
		$\widetilde P$ and $\widetilde \Psi$.
		By the definition \eqref{eq:bullet_Phi}
		of $\overset{\bullet}{\Psi}$,
		we obtain 
		\begin{align*}
			Q_{\overset{\bullet}{\Psi}}(w)&= 
			\sum_{i,j =0}^{L-1} \big( (w^{(i+1)})^{\top}
			\Psi_{i,j} w^{(j)}  + 
			(w^{(i)})^{\top}
			\Psi_{i,j} w^{(j+1)} \big)  \\
			&=
			\begin{bmatrix}
				w' \\ \vdots \\ w^{(L)}
			\end{bmatrix}^{\top}  \widetilde \Psi
			\begin{bmatrix}
				w \\ \vdots \\ w^{(L-1)}
			\end{bmatrix}  +
			\begin{bmatrix}
				w \\ \vdots \\ w^{(L-1)}
			\end{bmatrix}^{\top}  \widetilde \Psi
			\begin{bmatrix}
				w' \\ \vdots \\ w^{(L)}
			\end{bmatrix}.
		\end{align*}
		Since
		\begin{align*}
			\begin{bmatrix}
				w' \\ \vdots \\ w^{(L)}
			\end{bmatrix}
			&=
			\begin{bmatrix}
				J_{q(L-1)} \vspace{3pt}\\ -\widetilde P
			\end{bmatrix}
			\begin{bmatrix}
				w \\ \vdots \\ w^{(L-1)}
			\end{bmatrix},
		\end{align*}
		it follows that
		\begin{align}
			\label{eq:Q_bPsi}
			Q_{\overset{\bullet}{\Psi}}(w) &=
			\begin{bmatrix}
				w \\ \vdots \\ w^{(L-1)}
			\end{bmatrix}^{\top}
			\left( 
			\begin{bmatrix}
				J_{q(L-1)} \vspace{3pt}\\ -\widetilde P
			\end{bmatrix}^{\top}
			\widetilde \Psi 
			+
			\widetilde \Psi \begin{bmatrix}
				J_{q(L-1)} \vspace{3pt}\\ -\widetilde P
			\end{bmatrix}
			\right)
			\begin{bmatrix}
				w \\ \vdots \\ w^{(L-1)}
			\end{bmatrix}.
		\end{align}
		
		a)
		Suppose that \eqref{eq:JRPsi_ineq1} holds. By
		\eqref{eq:Q_bPsi},
		$-Q_{\overset{\bullet}{\Psi}}(w) \geq 0$ for all
		$w \in \mathcal{B}$.
		Moreover, if 
		$w \in \mathcal{B}$ satisfies
		$Q_{\overset{\bullet}{\Psi}}(w) = 0$, then
		$w(t) = 0$ for all $t \in \mathbb{R}$.
		Therefore, $-\overset{\bullet}{\Psi} \gB 0$.
		
		b)
		Suppose that $\overset{\bullet}{\Psi} - \Psi_0 \eB 0$.
		Then \eqref{eq:Q_bPsi} yields
		\[
		\begin{bmatrix}
			w \\ \vdots \\ w^{(L-1)}
		\end{bmatrix}^{\top}
		\left( 
		\begin{bmatrix}
			J_{q(L-1)} \vspace{3pt}\\ -\widetilde P
		\end{bmatrix}^{\top}
		\widetilde \Psi 
		+
		\widetilde \Psi \begin{bmatrix}
			J_{q(L-1)} \vspace{3pt}\\ -\widetilde P
		\end{bmatrix} + I_{qL}
		\right)
		\begin{bmatrix}
			w \\ \vdots \\ w^{(L-1)}
		\end{bmatrix} = 0
		\]
		for all $w \in \mathcal{B}$.
		For any
		$w_0,\dots,w_{L-1} \in \mathbb{R}^q$,
		we can construct $w \in \mathcal{B}$ satisfying
		$w^{(\ell)}(0) = w_{\ell}$ for all $\ell=0,\dots,L-1$.
		It follows that 
		\[
		\begin{bmatrix}
			-J_{q(L-1)} \vspace{3pt}\\ \widetilde P
		\end{bmatrix}^{\top}
		\widetilde \Psi 
		+
		\widetilde \Psi \begin{bmatrix}
			-J_{q(L-1)} \vspace{3pt}\\ \widetilde P
		\end{bmatrix} = I_{qL}.
		\]
		Thus, \eqref{eq:JRPsi_ineq1} holds.
	\end{proof}
	
	We are now in a position to prove Lemma~\ref{lem:QMI}.
	\begin{proof}[Proof of Lemma~\ref{lem:QMI}]
		Define 
		the system $\mathcal{B}$ 
		by
		\eqref{eq:closed_loop_behavior}.
		First, suppose 
		that there exists $\widetilde\Psi \in \mathbb{S}^{qL}$ such that 
		$\widetilde\Psi > 0$ and 
		\begin{equation}
			\label{eq:LMI_widetildePsi}
			\begin{bmatrix}
				-J_{q(L-1)} \\ C \\ R
			\end{bmatrix}^{\top}
			\widetilde\Psi + \widetilde\Psi 
			\begin{bmatrix}
				-J_{q(L-1)} \\ C \\ R
			\end{bmatrix} > 0.
		\end{equation}
		Choose $\Psi \in \Rsym$ with $\deg \Psi = L-1$
		such that 
		its block coefficient matrix is $\widetilde \Psi$.
		Since $\widetilde \Psi >0$, it follows that 
		$\Psi \geB  0$; see 
		\eqref{eq:B_nn}. By \eqref{eq:LMI_widetildePsi} and 
		Lemma~\ref{lem:Lyap_ineq}.a), we also obtain
		$-\overset{\bullet}{\Psi} \gB  0$.
		From 
		Theorem~\ref{thm:QDF}.a), it follows that
		$\mathcal{B}$ is stable; i.e., $C$ stabilizes $R$.	
		
		Next, suppose that 
		$C$ stabilizes $R$.
		By Theorem~\ref{thm:QDF}.b),
		there exists $\Psi \in \Rsym$ with $\deg \Psi = L-1$ such that 
		$\Psi \ggB 0$ and
		$\overset{\bullet}{\Psi} - \Psi_0 \eB 0$.
		Then by \eqref{eq:B_sp},
		the block coefficient matrix $\widetilde \Psi \in \mathbb{S}^{qL}$ of $\Psi$ satisfies 
		$\widetilde \Psi >0$. 
		The LMI \eqref{eq:LMI_widetildePsi} follows from
		Lemma~\ref{lem:Lyap_ineq}.b).
	\end{proof}
	
	\section{Approximation-based characterization of informativity for system identification}
	\label{sec:approximation}
	In this section, we characterize informativity for system identification
	using B-spline approximations of $\Hs$-functions.
	Section~\ref{sec:splines} collects
	basic properties on B-splines. We then use them in
	Section~\ref{sec:approximation_data_operator}
	to construct matrix approximations 
	of the data-embedded operator $H$ defined by \eqref{eq:H_def}.
	In Section~\ref{sec:sec:relation_informavitity}, we show that 
	exact system
	identification can be achieved for some finite approximation order,
	without letting the order go to infinity.
	
	\subsection{Basic facts on B-splines}
	\label{sec:splines}
	Let $\tau >0$ and $N \in \mathbb{N}$. Set $h_N \coloneqq \tau  2^{-N}$.
	For $n=0,\dots,2^{N}-1$,
	the 
	B-spline $B_{N,n}^0\colon [0,\tau] \to \mathbb{R}$ 
	of degree $0$ is defined by
	\[
	B_{N,n}^0(t) \coloneqq 
	\begin{cases}
		1, & nh_N \leq t < (n+1)h_N, \\
		0, & \text{otherwise}.
	\end{cases}
	\]
	For  $d \in \mathbb{N}$ and $n=0,\dots,2^N-d-1$, 
	the B-spline
	$B_{N,n}^d \colon [0,\tau] \to \mathbb{R}$
	of degree $d$
	is 
	defined recursively by
	\[
	B_{N,n}^d (t)\coloneqq 
	\frac{t-nh_N }{dh_N } B_{N,n}^{d-1} (t) + 
	\frac{(n+d+1)h_N - t}{dh_N } B_{N,n+1}^{d-1}(t),
	\quad t \in [0,\tau].
	\]
	Let  $d \in \mathbb{N}_0$ and $n=0,\dots,2^N-d-1$.
	The support of $B_{N,n}^d$, i.e., 
	the closure of the set
	\[
	\{t \in [0,\tau]: B_{N,n}^d(t) \neq 0\},
	\] 
	is given by $[nh_N, (n+d+1)h_N]$. 
	On each subinterval 
	$[ih_N, (i+1)h_N)$ for $i=n,\dots,n+d$, 
	$B_{N,n}^d$
	is a polynomial of degree $d$.
	Furthermore, $B_{N,n}^1$ is continuous on $[0,\tau]$, and if  $d \geq 2$, then 
	$B_{N,n}^d$
	is $(d-1)$-times continuously differentiable on $[0,\tau]$.
	
	The linear span of a subset $M$ of 
	a vector space $V$ is denoted by $\mathspan M$.
	For $N,L \in \mathbb{N}$,
	we define the finite-dimensional space
	$\Bs[0,\tau] $ by
	\[
	\Bs[0,\tau] \coloneqq 
	\begin{cases}
		\{0 \}, & 2^N < L+1. \\
		\mathspan\{ 
		B_{N,0}^L, B_{N,1}^L,\dots, B_{N,2^N-L - 1}^L
		\}, & 2^N \geq L+1.
	\end{cases}
	\]
	Then we have 
	\begin{equation}
		\label{eq:monotone_prop}
		\Bs[0,\tau] \subseteq \mathrm{B}_{N+1}^L[0,\tau]
	\end{equation}
	for all $N \in \mathbb{N}$;
	see, e.g., \cite[p.~527]{Schumaker2007}.
	The following lemma is useful to approximate the operator $H$.
	
	\begin{lemma}
		\label{lem:denseness}
		For all $L \in \mathbb{N}$, 
		$\bigcup_{N \in \mathbb{N}} \Bs[0,\tau]$ is 
		dense in $\Hs[0,\tau]$.
	\end{lemma}
	\begin{proof}
		First, we prove that $\Bs[0,\tau]
		\subseteq \Hs[0,\tau]$ for all $N,L \in \mathbb{N}$.
		Let $2^N \geq L+1$ and $n=0,\dots,2^N-L-1$.
		Since 
		the $L$-th derivative of $B_{N,n}^{L}$
		is constant on each subinterval $[ih_N,(i+1)h_N)$, it follows that 
		$B_{N,n}^{L} \in \Hso[0,\tau]$. 
		Moreover, if we extend 
		the domain of $B_{N,n}^{L}$ to $\mathbb{R}$ by setting
		$B_{N,n}^{L}(t) = 0$ for $t \in \mathbb{R} \setminus [0,\tau]$,
		then this extended function is also $(L-1)$-times continuously differentiable on $\mathbb{R}$.
		Therefore,
		$B_{N,n}^{L} \in \Hs[0,\tau]$.
		
		Let $\varepsilon >0$ and $f \in \Hs[0,\tau]$ be arbitrary.
		There exists an infinitely
		differentiable function $\phi$
		with compact support in $(0,\tau)$
		such that 
		$\|f - \phi \|_{\Hs} \leq \varepsilon$.
		Moreover,
		there exist $N \in \mathbb{N}$ and $\psi \in \Bs[0,\tau]$ such that $\|\phi - \psi \|_{\Hs} \leq \varepsilon$; see, e.g., 
		\cite[Corollary~6.26]{Schumaker2007}.
		Therefore, $\|f - \psi \|_{\Hs} \leq 2\varepsilon$.
		Since $\varepsilon >0$ is arbitrary, 
		we conclude that 
		$\bigcup_{N \in \mathbb{N}} \Bs[0,\tau]$ is 
		dense in $\Hs[0,\tau]$ for all $L \in \mathbb{N}$.
	\end{proof}
	
	\subsection{Approximation of data-embedded operators}
	\label{sec:approximation_data_operator}
	Let $N,L\in \mathbb{N}$ satisfy $2^N \geq L+1$.
	For $n=0,\dots,2^N-L-1$, let 
	$B_{N,n}^{L,k}\colon [0,\tau_k] \to \mathbb{R}$ 
	be the B-spline of order $L$ as
	defined in Section~\ref{sec:splines}.
	We define 
	the matrices
	$\widetilde{H}_{k,N} \in \mathbb{R}^{qL \times (2^N-L)}$
	and
	$\widetilde{H}_{N} \in \mathbb{R}^{qL \times K(2^N-L)}$
	by
	\begin{equation}
		\label{eq:tildeH_N}
		\widetilde{H}_{k,N} \coloneqq 
		\begin{bmatrix}
				H_k B_{N,0}^{L,k} & \cdots & H_kB_{N,2^N-L-1}^{L,k}
			\end{bmatrix} \quad 
		\text{and} \quad 
		\widetilde{H}_{N} \coloneqq 
		\begin{bmatrix}
				\widetilde{H}_{1,N} & \cdots & \widetilde{H}_{K,N}
			\end{bmatrix}.
	\end{equation}
	These matrices
	$\widetilde{H}_{k,N} $
	and
	$\widetilde{H}_{N}$ are 
	finite-dimensional approximations of
	the operators $H_k \in \mathcal{L}(\Hs[0,\tau],\mathbb{R}^{qL})$ 
	and 
	$H \in \mathcal{L}(\Ht,\mathbb{R}^{qL})$, respectively,
	defined in \eqref{eq:Hi_def} and \eqref{eq:H_def}.
	Similarly, 
	we define the matrices
	$\widetilde{Y}_{L,k,N}\in \mathbb{R}^{p \times (2^N-L)}$
	and
	$\widetilde{Y}_{L,N} \in \mathbb{R}^{p \times K(2^N-L)}$
	by
	\begin{equation}
		\label{eq:tildeY_N}
		\widetilde{Y}_{L,k,N} \coloneqq 
		\begin{bmatrix}
				Y_{L,k} B_{N,0}^{L,k} & \cdots & Y_{L,k}B_{N,2^N-L-1}^{L,k}
			\end{bmatrix} \quad 
		\text{and} \quad 
		\widetilde{Y}_{L,N} \coloneqq \begin{bmatrix}
				\widetilde{Y}_{L,1,N} & \cdots & \widetilde{Y}_{L,K,N}
			\end{bmatrix},
	\end{equation}
	which approximate
	the $L$-th synthesis operator $Y_{L,k}$
	associated with the $k$-th output $y_k$ and
	the operator $Y_L$ introduced in
	\eqref{eq:YL_def}, respectively.
	The following lemma shows that 
	the ranges of the operator $H$
	and the matrix $\widetilde{H}_{N}$
	coincide when $N$ is sufficiently large.
	\begin{lemma}
		\label{lem:range_H}
		For $N,L \in \mathbb{N}$ with $2^N \geq L+1$,
		define  $H \in \mathcal{L}(\Ht,\mathbb{R}^{qL})$
		and $\widetilde{H}_{N} \in \mathbb{R}^{qL \times K(2^N-L)}$ by
		\eqref{eq:H_def} and
		\eqref{eq:tildeH_N}, respectively. Then
		for all $L\in \mathbb{N}$,
		there exists $N \in \mathbb{N}$ such that 
		$\Ran H = \Ran \widetilde{H}_{N}$.
	\end{lemma}
	\begin{proof}
		Let $L\in \mathbb{N}$ and 
		let $H_k|_{\Bs}$ denote 
		the restriction of $H_k$ to $\Bs[0,\tau_k]$
		for $k=1,\dots,K$.
		First, we prove that 
		for each $k=1,\dots,K$,
		there exists $N_k \in \mathbb{N}$ such that 
		\begin{equation}
				\label{eq:RanHk_app}
				\Ran H_k = \Ran H_k|_{\mathrm{B}_{N_k}^L}.
			\end{equation}
		Fix $k \in \{1,\dots,K\}$.
		To prove the inclusion $\Ran H_k \subseteq \Ran H_k|_{\mathrm{B}_{N_k}^L}$, 
		let $\{ e_i:i=1,\dots,i_0\}$ be an
		orthonormal basis of $\Ran H_k$.
		For each $i=1,\dots,i_0$,
		there exists $\phi_i \in \Hs [0,\tau_k]$ such that
		$e_i= H_k \phi_i$.
		By the density of
		$\bigcup_{N \in \mathbb{N}} \Bs[0,\tau_k]$ in $\Hs[0,\tau_k]$  and the monotone property~\eqref{eq:monotone_prop},
		there exists $N_k \in \mathbb{N}$ such that, 
		for every $i=1,\dots,i_0$,
		$\mathrm{B}_{N_k}^L$
		contains an element $\psi_i$ satisfying
		\[
		\|\phi_i - \psi_i\|_{\Hs } < \frac{1}{\sqrt{i_0} \|H_k\|}.
		\]
		Then
		$\|e_i - H_k \psi_i \| < 1/\sqrt{i_0}$ for all $i=1,\dots,i_0$.
		This implies that 
		$\{ H_k \psi_i:i=1,\dots,i_0\}$ is also a basis of $\Ran H_k$;
		see, e.g., \cite[Exercise~6.B.6]{Axler2024}.
		Hence,
		the inclusion $\Ran H_k \subseteq \Ran H_k|_{\mathrm{B}_{N_k}^L}$ holds.
		Since the converse inclusion  is trivial, 
		it follows that \eqref{eq:RanHk_app} holds.
		
		Define $N \coloneqq \max\{N_1,N_2,\dots,N_K\}$.
		By construction, we  obtain
		\[\Ran H_k|_{\Bs} = \Ran \widetilde{H}_{k,N}
		\]
		for all $k=1,\dots,K$.
		Since the monotone property~\eqref{eq:monotone_prop}
		yields $
		\mathrm{B}_{N_k}^L[0,\tau_k] \subseteq
		\Bs[0,\tau_k]$ for all $k=1,\dots,K$,
		we deduce from \eqref{eq:RanHk_app} that
		\[
		\Ran H_k = \Ran H_k|_{\Bs} =  \Ran \widetilde{H}_{k,N}
		\]
		for all $k=1,\dots, K$. Thus,
		\[
		\Ran H = \left\{
		\sum_{k=1}^{K} b_k : b_k \in \Ran \widetilde{H}_{k,N},~
		k=1,\dots,K
		\right\} = \Ran \widetilde{H}_{N}
		\]
		is obtained.
	\end{proof}

	\subsection{Relation between informativity for 
			system identification and
			surjectivity of $\widetilde{H}_{N}$}
	\label{sec:sec:relation_informavitity}
	Finally, we relate the surjectivity of 
	the matrix $\widetilde{H}_N$ to
	informativity 
	for system identification.
	Intuitively, the identification error of 
	the system constructed from the approximation matrices
	$\widetilde{H}_N$ and  $\widetilde{Y}_{L,N}$
	tends to zero as $N \to \infty$. 
	However, the following proposition shows
	a stronger result that
	the error vanishes for some 
	finite $N \in \mathbb{N}$.
	\begin{proposition}
		\label{prop:identification_app}
		Assume that for the data $\mathfrak{D} = (u_k,y_k)_{k=1}^{K}$,
		there exists a system 
		$R_s \in \mathbb{R}^{p\times qL}$  such that 
		\eqref{eq:system_k_noise_free} with $R=R_s$ holds.
		Define the matrices
		$\widetilde{H}_{N},\widetilde{Y}_{L,N}$
		by 
		\eqref{eq:tildeH_N} and 
		\eqref{eq:tildeY_N}, respectively.
		Then
		the following statements are equivalent:
		\begin{enumerate}
				\renewcommand{\labelenumi}{\textup{(\roman{enumi})}}
				\item The data $(u_k,y_k)_{k=1}^{K}$
				are informative for system identification.
				\item There exists $N \in \mathbb{N}$
				such that $\widetilde{H}_N$ is surjective. 
			\end{enumerate}
		Moreover, if 
		statement~\textup{(ii)} holds, then
		$R_s = -\widetilde{Y}_{L,N} \widetilde{S}$
		for any right inverse $\widetilde{S}
		\in \mathbb{R}^{K(2^N-L)\times qL}$
		of $\widetilde{H}_N$.
	\end{proposition}

	\begin{proof}
		By Proposition~\ref{prop:identification},
		the data are informative for system identification if and only if
		the data-embedded opeator $H$ defined by
		\eqref{eq:H_def} is surjective.
		From this and 
		Lemma~\ref{lem:range_H},
		the equivalence (i) $\Leftrightarrow$ (ii)
		follows.
		
		Next, suppose that statement~(ii) holds.
		We  denote by $H|_{\Bs}$ 
		and $Y_L|_{\Bs}$
		the restrictions of $H$ and $Y_L$ to
		the orthogonal direct sum
		of $(\Bs[0,\tau_k])_{k=1}^K$,
		respectively.
		We partition a right inverse 
		$\widetilde{S}$ 
		of  $\widetilde{H}_N$ as 
		\[
		\widetilde{S}
		=
		\begin{bmatrix}
				\widetilde{S}_1 \\ \vdots \\ \widetilde{S}_{K}
			\end{bmatrix},
		\]
		where 
		$\widetilde{S}_1,\dots,\widetilde{S}_{K} \in \mathbb{R}^{(2^N-L) \times qL}$.
		Define 
		\[
		S b \coloneqq 
		\begin{bmatrix}
				\sum_{\ell = 1}^{2^N-L}
				(
				\widetilde{S}_1 b
				)_{\ell}  B_{N,\ell-1}^{L,1}  \\
				\vdots \\
				\sum_{\ell = 1}^{2^N-L}
				(
				\widetilde{S}_{K} b
				)_{\ell}  B_{N,\ell-1}^{L,K} 
			\end{bmatrix},\quad 
		b  \in \mathbb{R}^{qL},
		\]
		where $(
		\widetilde{S}_k b
		)_{\ell}$ is the $\ell$-th element of 
		$\widetilde{S}_k b\in \mathbb{R}^{2^N-L}$ for each
		$k=1,\dots,K$.
		Then $S$ is a 
		right inverse of $H|_{\Bs}$. Indeed,
		\begin{align}
				\label{eq:right_inverse}
				H|_{\Bs} S b =
				\sum_{k=1}^{K} 
				\sum_{\ell = 1}^{2^N-L}
				(
				\widetilde{S}_k b
				)_{\ell}  H_k B_{N,\ell-1}^{L,k} 
				=
				\sum_{k=1}^{K} 
				\widetilde{H}_{k,N} \widetilde{S}_k b =
				\widetilde{H}_{N} Sb
				= b
			\end{align}
		for all $b \in \mathbb{R}^{qL}$.
		Since $R_sH|_{\Bs} +Y_L|_{\Bs} = 0$ by Lemma~\ref{lem:data_consistency_cond},
		a calculation similar to \eqref{eq:right_inverse}
		shows that 
		\begin{align*}
				R_sb = -Y_L|_{\Bs}S b =
				- \sum_{k=1}^{K}
				\sum_{\ell = 1}^{2^N-L}
				(\widetilde{S}_k b)_{\ell} Y_{L,k} B_{N,\ell-1}^{L,k} 
				= 	- \sum_{k=1}^{K}
				\widetilde{Y}_{L,k,N}\widetilde{S}_k b 
				=
				- \widetilde{Y}_{L,N} Sb
			\end{align*}
		for all $b \in \mathbb{R}^{qL}$.
	\end{proof}
\printbibliography
	\end{document}